\documentclass[10pt]{amsart}
\usepackage{mystyle}   
\usepackage[utf8]{inputenc}
\usepackage{amsfonts}
\usepackage{graphics}
\usepackage[all, cmtip]{xy}
\usepackage{amsthm,amsfonts,amssymb,amsmath,amsxtra}
\usepackage[all]{xy}
\usepackage{flexisym}
\usepackage{mathrsfs}
\usepackage{enumitem}
\usepackage[dvipsnames]{xcolor}
\usepackage[colorlinks]{hyperref}
\hypersetup{
  citecolor   = blue 
 }

\usepackage{blkarray}
\newcommand{\matindex}[1]{\mbox{\scriptsize#1}}
\newcommand{\xdownarrow}[1]{%
  {\left\downarrow\vbox to #1{}\right.\kern-\nulldelimiterspace}
}
 \newcommand{\quash}[1]{}  

\newcommand{\Mpm}{{\rm M}^{\pm}}
\newcommand{\Mnaive}{{\rm M}^{\rm naive}}
\newcommand{\Mbl}{{\rm M}^{\rm bl}}
\newcommand{\ad}{{\rm ad}}
\newcommand{\sgn}{{\rm sgn}}

\newcommand{\diag}{{\rm diag}}

\newcommand{\End}{{\rm End}}
\newcommand{\Hom}{{\rm Hom}}
\newcommand{\Aut}{{\rm Aut}}
\newcommand{\Res}{{\rm Res}}
\newcommand{\Gal}{{\rm Gal}}

\newcommand{\etale}{\' etale }

\newcommand{\Tr}{{\rm Tr}}

\newcommand{\Gr}{{\rm Gr}}

\newcommand{\Sets}{{\rm Sets}}

\newcommand{\GL}{{\rm GL}}

\newcommand{\GO}{{\rm GO}}
\newcommand{\OGr}{{\rm OGr}}

\newcommand{\SO}{{\rm SO}}

\newcommand{\calF}{\mathcal{F}}

\newcommand{\calB}{\mathcal{B}}

\newcommand{\calS}{\mathcal{S}}
\newcommand{\calI}{\mathcal{I}}

\newcommand{\calU}{\mathcal{U}}

\newcommand{\bb }{\langle}
\newcommand{\pp}{\rangle}

\newcommand{\ov}{\overline}

\usepackage{multicol}
\numberwithin{equation}{subsection}

\newtheorem{Main Theorem}[Theorem]{Main Theorem}

\usepackage[colorlinks]{hyperref}
\makeatletter
\renewcommand*\env@matrix[1][*\c@MaxMatrixCols c]{%
  \hskip -\arraycolsep
  \let\@ifnextchar\new@ifnextchar
  \array{#1}}
\makeatother

\newif\ifgrading

\gradingfalse

\usepackage[colorlinks]{hyperref}
\hypersetup{linkcolor=black}
\usepackage{xcolor}
\usepackage{ wasysym }

\usepackage{tikz-cd}
\newcommand{\U}{{\mathcal U}}

\newcommand{\Spec}{{\rm Spec \, } }

\usepackage{xcolor}

\newcommand{\Addresses}{{
		\bigskip
		\footnotesize

        \textsc{Yau Mathematical Sciences Center, Tsinghua University,
			Beijing, 100084, China}\par\nopagebreak
		\textit{E-mail address:} 
        \texttt{jie-yang@mail.tsinghua.edu.cn}\\
        
		\textsc{Department of Mathematics, Universität Münster, Münster, 48149, Germany}\par\nopagebreak
		\textit{E-mail address:} \texttt{io.zachos@uni-muenster.de}\\
		
		\textsc{Department of Mathematics, University of Science and Technology Beijing,
			Beijing, 100083, China}\par\nopagebreak
		\textit{E-mail address:} \texttt{zhihaozhao@ustb.edu.cn}
}}	

\begin{document}

\title{
On $p$-adic integral moduli schemes and local models for PEL type D
}
	\date{}
	\author{Jie Yang, Ioannis Zachos and Zhihao Zhao}

\begin{abstract}
We construct flat integral moduli schemes of PEL type D and the corresponding flat orthogonal Rapoport--Zink spaces with parahoric level structure over a $p$-adic integer ring. The construction relies on proving a conjecture of Pappas--Rapoport: for an even orthogonal similitude group over a complete discretely valued field of residue characteristic $p>2$, and for arbitrary parahoric level, the associated spin local model is flat, normal, Cohen--Macaulay, with reduced special fiber. In the course of the proof, we also show that in the quasi-split but non-split case, the Rapoport--Zink (naive) local model is topologically flat, verifying a conjecture of Pappas--Rapoport--Smithling \cite[Conjecture~2.7.1]{PRS}. In the maximal parahoric case, we also describe the Schubert varieties in the special fiber in moduli-theoretic terms. Finally, for a maximal parahoric case we construct an explicit regular semi-stable model by blowing up the spin local model along the unique closed Schubert cell in its special fiber.
\end{abstract}

	\maketitle
	\tableofcontents

\section{Introduction}\label{Intro}
Local models are projective schemes over the ring of integers of a $p$-adic field which are designed to control the local geometry of integral models of Shimura varieties with parahoric level structure.  Starting from the pioneering works of Chai--Norman \cite{CN90}, Deligne--Pappas \cite{DelignePappas94} and of de Jong \cite{deJ}, the theory was formalized to some extent in the book of Rapoport--Zink \cite{RZbook} for Shimura varieties of PEL type; in particular, local models are defined there in a linear-algebraic manner in terms of moduli of self-dual lattice chains in skew-Hermitian spaces. When the underlying group is of type A or C and splits over an unramified extension, G\"ortz proved that these local models are flat with reduced special fiber \cite{Go,Goertz03}. In general, the resulting schemes are not always flat: Pappas first observed this phenomenon for ramified unitary groups \cite{P}, and Genestier later noted analogous failures for split orthogonal groups; hence the Rapoport--Zink local models are nowadays called naive local models; see the survey \cite{PRS} for details and further references.

In \cite{PZ}, Pappas and Zhu gave a new construction of local models for tamely ramified groups which, unlike the earlier lattice-theoretic constructions, is purely group-theoretic (in particular, it is not tied to a particular representation of the group). The construction has been further refined in subsequent work, for instance by He--Pappas--Rapoport \cite{HPR}; see also Levin \cite{BL} and Louren\c{c}o \cite{Lour} for extensions and refinements. These local models are flat by construction, and in great generality they are known to be normal and Cohen--Macaulay with reduced special fiber (see \cite{haines2022normality}). 
Moreover, in the tamely ramified setting they are ``canonical'' in the sense of Scholze--Weinstein: this is proved in many cases in \cite[Corollary 2.17]{HPR} and, in the general case, in \cite{AGLR22}. On the other hand, this group-theoretic construction does not in general provide a moduli interpretation. It therefore remains an important problem to give a moduli-theoretic description of the canonical local model, which is crucial for many arithmetic applications. In the PEL setting, it is known that the canonical local model is a closed subscheme of the naive local model; in particular it can be characterized as the scheme-theoretic closure of the common generic fiber (see \cite[Lemma 4.1]{HPR}). The task is thus to modify the Rapoport--Zink moduli functor by additional linear-algebraic conditions so that it represents the canonical local model.

In the ramified unitary case this moduli-description problem is by now well understood in several important situations. As a first step, Pappas \cite{P} imposed the wedge condition on the naive moduli functor and in the self-dual maximal parahoric case for signature $(n-1,1)$, this condition already cuts out 
the local model. For general parahoric level the wedge condition is not sufficient, and Pappas--Rapoport \cite{PR} introduced the spin condition; subsequent refinements led to the strengthened spin condition defined by Smithling \cite{Sm3}. For signature $(n-1,1)$, Luo \cite{Luo} proved that the strengthened spin condition cuts out the canonical local model for arbitrary parahoric level.

For even orthogonal (type D) cases, the analogous moduli-description problem is more subtle. Pappas--Rapoport \cite{PR} introduced the \emph{spin} local model, obtained from the naive model by imposing a spin
condition, and conjectured that it is flat for arbitrary parahoric level \cite[Conjecture~8.1]{PR}. In the split case, the first author
\cite{Yang25,yang25TopFlat} proved the corresponding flatness result for the
spin local model for arbitrary parahoric level.


In this paper we treat the quasi-split, non-split even orthogonal similitude group.
This case has not previously been treated in the literature; moreover, our results apply to every parahoric level. We show that imposing the spin condition yields a flat local model with reduced special fiber; equivalently,
the spin local model agrees with the canonical local model. This proves \cite[Conjecture~8.1]{PR} in our quasi-split setting. Together
with the split case, this proves the Pappas--Rapoport flatness conjecture for the
even orthogonal groups arising in the PEL type \(D\) setting. A key geometric feature of the quasi-split non-split case is an irreducibility
phenomenon for the special fiber in the maximal parahoric situation
(Theorem~\ref{intro.str}), which is crucial for our reducedness argument; this contrasts with the split case, where the analogous special fiber has two irreducible components. In the maximal parahoric case, we also give an explicit moduli-theoretic description of the Schubert varieties appearing in the special fiber. 
In the course of the proof we also establish that the naive local model is topologically flat, thereby verifying \cite[Conjecture~2.7.1]{PRS}. As arithmetic applications, we show that the resulting spin condition also yields a normal, Cohen–Macaulay, flat integral moduli space of PEL type D with reduced special fiber, and that the associated Shimura variety is an open and closed subscheme of its generic fiber; moreover, we obtain an explicit moduli interpretation of the corresponding orthogonal Rapoport–Zink space. Finally, for a specific maximal parahoric level, using the resulting moduli description of the spin local model we show that the blow-up along the unique closed Schubert stratum in the special fiber yields an explicit regular semi-stable model.
We expect that such explicit semi-stable models will find applications to arithmetic intersections problems (for instance, in the context of Kudla’s program).

\subsection{Main Results}\label{MainResults} 
Let $F_0$ be a complete discretely valued field with ring of integers $\CO_{F_0}$ and residue field $k$ of characteristic $p> 2$ and uniformizer $\pi_0$. 
Let $V$ be a $2n+2$-dimensional $F_0$-vector space equipped with a non-degenerate symmetric paring $\phi$. Let $G=\GO(V,\phi)$ denote the orthogonal similitude group over $F_0$. Since $p>2$, after replacing $F_0$ by a sufficiently large unramified extension, we may assume (cf. \cite[\S 2.7]{PRS} that there exists an ordered $F_0$-basis $e_1,\ldots,e_{2n},f_1,f_2$ of $V$ such that the corresponding matrix form of $\phi$ is one of the following two forms:
\begin{flalign*}
    \begin{pmatrix}
        H_{2n} \\ &{\begin{pmatrix}
             &1 \\ 1
        \end{pmatrix}}  
    \end{pmatrix}  \text{\ or\ } \begin{pmatrix}
        H_{2n} \\ &{\begin{pmatrix}
            \pi_0\\ &1
        \end{pmatrix}} 
    \end{pmatrix},
\end{flalign*}
where $H_{2n}$ is the unit anti-diagonal matrix of size $2n$.
The first case corresponds to the split orthogonal similitude group, which is treated in \cite{Yang25}. From now on, we focus on the second case corresponding to the quasi-split but non split $G$. We further choose $F/F_0$ to be a ramified quadratic extension and $\pi \in F$ a uniformizer with $\pi^2 = -\pi_0$.

Denote by $G^\circ$ the connected component of $G$ containing the identity element. Let $\CL$ be a
self-dual periodic lattice chain of $V$ in the sense of \cite[Corollary 2.17]{RZbook}. Denote by
$\sG_{\CL}$ the (affine smooth) group scheme of similitude automorphisms of $\CL$. By Bruhat--Tits
theory, the neutral component $\sG^\circ_{\CL}$ is a parahoric group scheme for $G^\circ$. In the
lattice-theoretic framework of Rapoport--Zink \cite{RZbook} one associates to $\CL$ a projective
$\CO_{F_0}$-scheme $\RM^\naive_\CL$, the \emph{naive local model}, representing a linear-algebraic
moduli problem reviewed in \S\ref{sec-naivespinmod}.  Its generic
fiber is the orthogonal Grassmannian $\OGr(n+1,V)$ of maximal totally isotropic $(n+1)$-dimensional subspaces in $(V,\phi)$. The base change $\RM^\naive_{\CL,F_0}\otimes_{F_0}F$ consists of two connected components $\OGr(n+1,V)_\pm$, which are isomorphic to each other, and can be identified with the flag variety $G^\circ_F/P_{\mu_\pm}$, where $P_{\mu_\pm}$ denotes the parabolic subgroup (over $F$) associated to a minuscule cocharacter $\mu_\pm$ given by
\[
    \mu_+ := (1^{(n+1)},0^{(n+1)}) \text{\ and\ } \mu_- := (1^{(n)},0,1,0^{(n)}).
\]
After base change to $\CO_{F}$, one can refine the moduli problem by imposing the
\emph{spin condition} of Pappas--Rapoport, which cuts out a closed subscheme
\[
\RM^\pm_\CL \;\subset\; \RM^\naive_\CL\otimes_{\CO_{F_0}}\CO_{F},
\]
called the \emph{spin local model}, see Definition \ref{defn-spin}. Note that the generic fiber of $\RM^\pm_\CL$ is one of the two components $\OGr(n+1,V)_\pm$. Denote by $\RM^{\pm\rm loc}_\CL$ the schematic closure of $\RM^\pm_\CL\otimes F$ in $\RM^\pm_\CL$. Pappas--Rapoport conjectured that the spin condition already cuts out this closure.

\begin{conj}[{\cite[Conjecture 8.1]{PR}}] \label{introconj-pr}
    The spin local model $\RM^\pm_\CL$ is flat over $\CO_{F}$. 
\end{conj}
In this context, Pappas--Rapoport--Smithling formulated a second conjecture concerning
topological flatness of the naive local model.

\begin{conj}[{\cite[Conjecture 2.7.1]{PRS}}] \label{introconj-prs}
The naive local model $\RM^\naive_\CL$ is topologically flat, i.e., the generic fiber $\RM^\naive_{\CL,F_0}$ is Zariski dense in $\RM^\naive_\CL$.
\end{conj}
The main result of this article proves both conjectures.

\begin{thm}\label{mainthmintro}
{\renewcommand{\labelenumi}{(\arabic{enumi})}%
\begin{enumerate}
  \item (Theorem~\ref{thmflatspin}) The scheme $\RM^\pm_\CL$ is $\CO_{F}$-flat of relative dimension $n(n + 1)/2$, normal, and Cohen--Macaulay with reduced special fiber.
  \item (Theorem~\ref{coro-mainresults}) The scheme $\RM^\naive_\CL$ is topologically flat.
\end{enumerate}}
\end{thm}

\begin{remark}\label{rem:PZ}
In the present PEL setting, $\RM^{\pm\rm loc}_\CL$ coincides with the corresponding canonical schematic local model as well as the Pappas--Zhu local model;
see Proposition~\ref{propiso}. Hence Theorem~\ref{mainthmintro}(1) yields a moduli-theoretic description of the
Pappas--Zhu local model.
\end{remark}

\subsection{Outline of the proof} We briefly outline the proof of Theorem~\ref{mainthmintro}.
By Remark~\ref{rem:PZ}, once flatness is proved, $\RM^\pm_\CL$ identifies with the corresponding
Pappas--Zhu local model. The additional properties in Theorem~\ref{mainthmintro}(1) then follow from the
general results on Pappas--Zhu local models recalled in \S\ref{sec-naivespinmod}. Hence the main step is to prove that
$\RM^\pm_\CL$ is $\CO_{F}$-flat; along the way, we also deduce the topological flatness of the naive local model.

A key input is a classification of the parahoric subgroups of $G$, which we establish next and use to reduce to the standard lattice-chain cases.

\begin{thm}[{Theorem~\ref{prop-paraIndex}}]\label{intro.prop-paraIndex}
Denote by $\Lambda_i \subset V$ the standard lattices where 
\[
\Lambda_i :={\CO_{F_0}} \bb \pi_0\inverse e_1,\ldots,\pi_0\inverse e_i,e_{i+1},\ldots,e_{2n},\pi_0\inverse f_1,f_2 \pp, \text{ for } 0\leq i \leq n.
\]
Consider the stabilizer subgroup
\[
P_I\coloneqq \{g\in G\ |\ g\Lambda_i=\Lambda_i, i\in I\}
\]
where $I$ is a non-empty subset of $[0,n] := \{0,\dots,n\}$, and set $P^{\circ}_I \coloneqq P_I\cap G^\circ(F_0)$.

\begin{enumerate}
\item The subgroup $P^{\circ}_I$ is a parahoric subgroup of $G(F_0)$.
Moreover, any parahoric subgroup of $G(F_0)$ is $G^\circ(F_0)$-conjugate to
$P^{\circ}_I$ for a unique non-empty subset $I\subset [0,n]$ with the property
that if $i\in I$ with $i\geq \lfloor n/2\rfloor$, then $n-i\in I$.

\item The $G^\circ(F_0)$-conjugacy classes of maximal parahoric subgroups of
$G^\circ(F_0)$ are in bijection with the set
\[
\{P^{\circ}_{\{i\}}\mid 0\leq i\leq \lfloor n/2\rfloor \},
\]
and $P^{\circ}_{\{0\}}$ corresponds to a special maximal parahoric. For
$i>\lfloor n/2\rfloor$, the subgroup $P_{\{i\}}^\circ$ is $G^\circ(F_0)$-conjugate
to $P_{\{n-i\}}^\circ$.
\end{enumerate}
\end{thm} 
Using the above proposition, we can reduce our study of the local models to $\RM^\pm_{\Lambda_I}$ ($\RM^\naive_{\Lambda_I}$ resp.) for the standard lattice chains $ \Lambda_I := \{\Lambda_{\ell}\}_{\ell\in 2n\BZ\pm I} $; see \S \ref{ParahoricSubsection} for details. To simplify the notation, we write $\RM^\pm_{I}$ ($\RM^\naive_{I}$ resp.) for $\RM^\pm_{\Lambda_I}$ ($\RM^\naive_{\Lambda_I}$ resp.). To show flatness for $\RM^\pm_{I}$, it is enough to prove:

(a) $\RM^\pm_{I}$ is topologically flat (see \S \ref{sec-topoflat}), and 

(b) the special fiber $\RM^\pm_{I,k}$ is reduced (see \S \ref{AffineCharts}, \ref{FlatSec}). 

In the course of proving (a) we obtain the following stronger statement.
\begin{thm} [{Theorem \ref{coro-mainresults}}] \label{intro-thmtopo}
The naive local model $\RM^\naive_{I}$, and hence $\RM^\pm_{I}$, is topologically flat for any non-empty $I\sset [0,n]$. 
\end{thm}
 
In particular, this implies Theorem~\ref{mainthmintro}(2). Using the theory of partial affine flag varieties and their relation to local models, we can reduce the proof of the above theorem to the maximal cases, i.e. to $I=\{i\}$; a key ingredient that we use for this reduction is the vertexwise criteria \cite{haines2017vertexwise} for the admissible sets. Thus it remains to treat the maximal case $ \Lambda_{\{i\}}$. In this case we obtain an explicit Schubert-cell stratification of the reduced special fiber:

\begin{thm}\label{intro.str}
Let $I = \{i\}$. There exists a stratification of the reduced special fiber \begin{flalign}
        (\RM^\naive_{I,k})_\red=\coprod_{\ell=\max\cbra{0,2i-n}}^i \RM^\naive_I(\ell), \label{intro.stratification}
    \end{flalign}
    where each stratum $\RM^\naive_I(\ell)$  is a single Schubert cell, described
explicitly in Definition~\ref{defn-Mell}. If $S_\ell:=\ol{\RM^\naive_I(\ell)}
        \subset (\RM^\naive_{I,k})_\red$,
then $S_\ell$ is the corresponding Schubert variety and
\[
        S_\ell
        =
        \coprod_{\ell'=\max\{0,2i-n\}}^\ell
        \RM^\naive_I(\ell').
\]
Equivalently, $\RM^\naive_I(\ell')$ is contained in
$\ol{\RM^\naive_I(\ell)}$ if and only if $\ell'\leq \ell$.
Consequently, $\RM^\naive_{I,k}$ is irreducible and contains precisely $\min{\cbra{i,n-i}}+1$ Schubert cells.    
\end{thm}
A crucial feature of Theorem~\ref{intro.str} is the irreducibility of the special fiber in the maximal case,
which plays an important role in our proof of Theorem~\ref{mainthmintro}. This contrasts with the split even orthogonal case, where in the corresponding maximal-type situation
the special fiber has two irreducible components; see \cite[Theorem~1.5(2),(3)]{yang25TopFlat}. We further give an explicit moduli-theoretic description of the Schubert
varieties $S_\ell$ occurring in Theorem~\ref{intro.str}.
\begin{prop}[{Proposition \ref{moduli-sv}}]\label{defn-schubert}
Assume that $I=\{i\}$. The associated standard lattice chain is 
\[
\Lambda_{-i}\subset \Lambda_i\subset \pi_0^{-1}\Lambda_{-i},
\]
where $\Lambda_{-i}=\Lambda^\vee_i$ denotes the dual lattice taken with respect to $\phi$.

For $\max\cbra{0,2i-n}\leq\ell\leq i$, the Schubert variety $S_\ell$ represents the functor which sends a $k$-algebra $R$ to the set of pairs $(\CF_i,\CF_{-i})$ satisfying the following conditions:
\begin{enumerate}
		\item $\CF_{\pm i}\sset \Lambda_{\pm i} \otimes R$ is, Zariski locally on $\Spec R$, a direct summand of rank $n+1$;
		\item $\phi(\CF_i,\CF_{-i})=0$;
		\item If $\iota_1:\Lambda_{-i}\to \Lambda_i$, $\iota_2:\Lambda_i\to \pi_0^{-1}\Lambda_{-i}$ denote the natural inclusion maps (or their base changes), then $\iota_1(\mathcal{F}_{-i})\subset \mathcal{F}_i$, $\iota_2(\mathcal{F}_i)\subset \pi_0^{-1}\mathcal{F}_{-i}$;
		\item $\wedge^{\ell+1}(\iota_2:\CF_i\ra \pi_0^{-1}\CF_{-i})=0$.
	\end{enumerate}
\end{prop}

We finish the proof of topological flatness by writing down explicit representatives for these cells and checking that each representative lifts to the generic fiber of $\RM^\naive_{\{i\}}$. It follows that the generic fiber is dense in $\RM^\naive_{\{i\}}$.

Next we address~(b), namely reducedness of the special fiber $\RM^\pm_{I,k}$. We first treat the case $I=\{i\}$ and then, using an argument similar to
\cite[\S 4.5]{Go}, we deduce reducedness for general $I$. In the maximal case $I=\cbra{i}$, we may assume that $0\leq 2i\leq n$ by Theorem \ref{intro.prop-paraIndex}(2). We compute an explicit affine chart $\CU^\pm_i\subset \RM^\pm_{\{i\}}$ 
containing a point $*$ in the minimal stratum of~\eqref{intro.stratification}; in particular, we obtain an explicit
description of its special fiber $\CU^\pm_{i,k}$. Set 
\[
\CR_i:=\frac{\CO_F[X]}{(X^tH_{2i+1}X,\; XH_{2i+1}X^t,\; \wedge^{i+1}X)},
\]
where $X$ is a $(2i+1)\times(2i+1)$ matrix and $H_{2i+1}$ is the unit anti-diagonal matrix of size $(2i+1)\times(2i+1)$. Let $\CR_{i,k}:=\CR_i\otimes_{\CO_F}k$.
\begin{thm}\label{intro:chart}
\begin{enumerate}
    \item (Lemma \ref{lem-Gtrans}) The $\sG^\circ_{\Lambda_i}$-translates of $ \CU^\pm_{i}$ cover $\RM^\pm_{\{i\}}$.
    \item (Corollary \ref{coro-Uspin}) The special fiber $\CU^\pm_{i,k}$ of $\CU^\pm_{i}$ is reduced and isomorphic to 
    \[
\mathbb{A}_k^{\frac{(n-2i)(n+2i+1)}{2}}\times \Spec \CR_{i,k}.
\] 
\end{enumerate}
\end{thm}

\begin{remark}\label{rem:i=0-smooth}
For $I=\{0\}$, Theorem~\ref{intro:chart}(2) gives
$\CU^\pm_{0,k}\cong \mathbb{A}^{\frac{n(n+1)}{2}}_k$.
In fact, Remark~\ref{ExoticGoodReduction} shows that $\RM^\pm_{\{0\}}$ is smooth over $\CO_{F}$. This is a case of ``exotic" good reduction (see also \cite[Theorem 1.2]{HPR}).
\end{remark}

By Theorem~\ref{intro:chart}(1), it suffices to show that the affine chart
$\CU^\pm_{i,k}$ is reduced. In \S \ref{FlatSec}, using the irreducibility of the special fibers of the above local models (Theorem \ref{intro.str}), we prove that $\CR_{i,k}$ is irreducible and generically smooth (hence generically reduced). On the other hand,
\cite[Theorem~1.11]{Yang25} shows that $\CR_i$ is $\CO_F$-flat and reduced.
Applying Hironaka's lemma and using the irreducibility of $\CR_{i,k}$, we deduce that $\CR_{i,k}$ is reduced. It follows that $\CU^\pm_{i,k}\cong \mathbb{A}_k^{\frac{(n-2i)(n+2i+1)}{2}}\times \Spec \CR_{i,k}$ is reduced, hence $\RM^\pm_{\{i\},k}$ is reduced. This completes the outline of the proof of Theorem~\ref{mainthmintro}.


\subsection{Applications}
We now briefly indicate some applications of the above results.

Assume $I=\{1\}$ and consider the blow-up $\rho:\Mbl_{\{1\}}\to \Mpm_{\{1\}}$ along the minimal stratum of~\eqref{intro.stratification}.
Using the explicit description of the chart $\CU^\pm_{1}$, we show:

\begin{thm}\label{intro:semistable}
Assume $I=\{1\}$. 
The blow-up $\Mbl_{\{1\}}$ is regular and its special fiber is a reduced divisor with two smooth
irreducible components intersecting transversely.
\end{thm}

As arithmetic applications, via the local model diagrams recalled in \S\ref{ShimuraVarSec}, Theorem~\ref{mainthmintro}
yields corresponding flatness results and moduli-theoretic descriptions for integral PEL moduli spaces of type~D and for the associated orthogonal Rapoport--Zink spaces. For the case $I= \{1\}$, 
we construct semi-stable models via the
above blow-up; see \S\ref{ShimuraVarSec}.

\subsection{Notation} \label{subsec-notation-intro}
Throughout the paper, $(F_0,\CO_{F_0},\pi_0, k)$ denotes a complete discretely valued field with ring of integers $\CO_{F_0}$ , uniformizer $\pi_0$, and residue field $k$ of characteristic $p>2$. Denote by $\ol{k}$ an algebraic closure of $k$. We also employ an auxiliary complete discretely valued field $K$ with ring of integers $\CO_K$, uniformizer $t$, and the same residue field $k$; eventually $K$ will be the field $k((t))$ of Laurent series over $k$.

For $1\leq i\leq 2n$, define $i^*\coloneqq 2n+1-i$. Given $v\in\BZ^{2n+1}$ and $1\leq j\leq 2n+1$, we write $v(j)$ for the $j$-th entry of $v$, and $\Sigma v$ for the sum of entries of $v$. Let $v^*\in\BZ^{2n+1}$ denote the vector defined by $v^*(i)\coloneqq v(i^*)$ for $1\leq i\leq 2n$ and $v^*(2n+1)\coloneqq -v(2n+1)-1$. For $v,w\in\BZ^{2n+1}$, we write $v\geq w$ if $v(j)\geq w(j)$ for all $1\leq j\leq 2n+1$. The expression $(a^{(r)},b^{(s)},\ldots)$ denotes the vector with $a$ repeated $r$ times, followed by $b$ repeated $s$ times, and so on. For $d\in \BZ$, set $\mathbf d_0\coloneqq (d^{(2n)},0) \in \BZ^{2n+1}$ and $\mathbf d\coloneqq (d^{(2n+1)})\in \BZ^{2n+1}$. 

For any real number $x$, denote by $\lfloor x\rfloor$ (resp. $\lceil x\rceil$) the greatest (resp. smallest) integer $\leq x$ (resp. $\geq x$).

For integers $n_2\geq n_1$, denote $[n_1,n_2]\coloneqq \cbra{n_1,\ldots,n_2}$.  For a subset $E\sset [1,2n]$ of cardinality $n$, set $E^*\coloneqq \cbra{i^*\ |\ i\in E}$ and $E^\perp\coloneqq (E^*)^c$ (the complement of $E^*$ in $[1,2n]$). Denote by $\Sigma E$ the sum of the entries in $E$.  

For a scheme $\CX$ over $\CO_{F_0}$, we let $\CX_k$ or $\CX\otimes k$ denote the special fiber $\CX\otimes_{\CO_{F_0}} k$, and $\CX_{F_0}$ or $\CX\otimes F_0$ denote the generic fiber $\CX\otimes_{\CO_{F_0}} F_0$.  

\smallskip

{\bf Acknowledgements:}  We thank B. Howard, Y. Luo and G. Pappas for helpful comments. J. Y. was supported by the Shuimu Tsinghua Scholar Program of Tsinghua
University. I. Z. was supported by Germany's Excellence Strategy EXC~2044--390685587 ``Mathematics M\"unster: Dynamics--Geometry--Structure'' and by the CRC~1442 ``Geometry: Deformations and Rigidity'' of the DFG. Z. Z. was supported by the Tianyuan Fund for Mathematics of NSFC, No. 12526543.

\section{Conjugacy classes of parahoric subgroups} \label{sec-para}
We use the notation as in \S \ref{subsec-notation-intro}. In particular, $K$ is a complete discretely valued field. Denote by $\val: K\cross\ra \BZ$ the normalized valuation on $K$ with $\val(t)=1$. Let $L=K(u)$ be a quadratic extension of $K$ with uniformizer $u$ satisfying $u^2=-t$. Denote by $\val_L:L\cross\ra \BZ$ the normalized valuation on $L$ with $\val(u)=1$. Let $\Gamma\coloneqq \Gal(L/K)$ be the Galois group with the unique nontrivial element $\gamma$ sending $u$ to $-u$. 

\subsection{Orthogonal similitude group}\label{OSG}
Let $n\geq 1$ be an integer. Let $V$ be a $K$-vector space of dimension $2(n+1)$ equipped with a quasi-split but non-split non-degenerate symmetric $K$-bilinear form $\phi$. By the classification
of such forms, after replacing $K$ by a sufficiently big unramified extension (still denoted by $K$), we may (and
do) choose an ordered $K$-basis $e_1,\ldots,e_{2n},f_1,f_2$ of $V$ such that the matrix of $\phi$
with respect to this basis is
\begin{flalign*}
\begin{pmatrix}
H_{2n} & \\ & \begin{pmatrix} t & \\ & 1 \end{pmatrix}
\end{pmatrix},
\end{flalign*}
where $H_{2n}$ denotes the $2n\times 2n$ anti-diagonal unit matrix; see, e.g. 
\cite[\S2.2]{I.Z} and \cite[\S2.7]{PRS}. Denote by $G=\GO(V,\phi)$ the associated orthogonal similitude group with similitude character \begin{flalign*}
	   \eta: G\ra \BG_m.
\end{flalign*}  Let $G^{\circ}$ be the connected component of $G$ containing the identity element. Then \begin{flalign*}
	  G^\circ=\cbra{g\in G\ |\ \eta(g)^{n+1}=\det(g) }.
\end{flalign*} Denote by $\tau\coloneqq \diag(1^{(2n)},-1,1)\in G\sset \GL_{2n+2}$. We have \begin{flalign}
	G = G^{\circ}\sqcup \tau G^{\circ}.  \label{Gtwocompo}
\end{flalign}   

Let $S$ be the maximal $K$-split torus of $G$. For a $K$-algebra $R$, we have \begin{flalign*}
	 S(R)=\cbra{\diag(x_1,\ldots,x_{2n}, y, y)\in\GL_{2n+2}(R)\ |\ x_1x_{2n}=x_2x_{2n-1}=\cdots=x_nx_{n+1}=y^2 }.
\end{flalign*}
The centralizer $T$ of $S$ in $G$ is a maximal torus (over $K$) of $G$, whose $R$-points are given by \begin{flalign*}
	\cbra{\diag(x_1,\ldots,x_{2n},\begin{psmallmatrix}
		    y_1 &y_2\\ -t y_2 &y_1
		\end{psmallmatrix})
	\in \GL_{2n+2}(R) \ \vline \ \begin{array}{l}
		x_1x_{2n}=\cdots=x_nx_{n+1}=y_1^2+t y_2^2 \\ x_1,\ldots,x_{2n},y_1,y_2\in R
	\end{array} }.
\end{flalign*}
We have $T\simeq \BG_{m,K}^{n}\times \Res_{L/K}\BG_{m,L}$ and $T\otimes_KL\simeq \BG_{m,L}^{n+2}$.
Denote \begin{flalign*}
	\varphi: \BG_{m,L}\times \BG_{m,L} &\simto \cbra{A\in \GL_{2,L}\ |\ A^t\begin{psmallmatrix}
		t \\ &1
	\end{psmallmatrix}A=\det(A)\begin{psmallmatrix}
		t \\ &1
	\end{psmallmatrix} } \\ (z_1,z_2) &\longmapsto  \begin{pmatrix}
		\frac{z_1+z_2}{2} &\frac{z_1-z_2}{2u}\\ -t\frac{z_1-z_2}{2u} &\frac{z_1+z_2}{2}
	\end{pmatrix}.
\end{flalign*}
The inverse of $\varphi$ is given by sending $\begin{psmallmatrix}
	y_1 &y_2\\ -t y_2 &y_1
\end{psmallmatrix}$ to $(y_1+u y_2,y_1-u y_2)$.
We have \begin{flalign*}
	T(L) = \cbra{\diag(x_1,\ldots,x_{2n},\varphi(z_1,z_2))\in \GL_{2n+2}(L)\ \vline \ \begin{array}{l}
		x_1x_{2n}=\cdots=x_nx_{n+1}=z_1z_2 \\ x_1,\ldots,x_{2n},z_1,z_2\in L
	\end{array}  }.
\end{flalign*}
The group $X_*(T)=\Hom_L(\BG_{m,L},T_L)$ of (geometric) cocharacters of $T$ is isomorphic to $\BZ^{n+2}$. More explicitly, 
we have an isomorphism \begin{flalign}
    \begin{split}
    	T(L)/T(\CO_L) &\simto X_*(T)\\ \diag(x_1,\ldots,x_{2n},\varphi(z_1,z_2)) &\mapsto (\val_L(x_1),\ldots,\val_L(x_{n}),\val_L(z_1),\val_L(z_2) ).  \label{Tlattice}
    \end{split}
\end{flalign}
The element $\gamma\in\Gamma$ acts on $X_*(T)$ by sending $(a_1\ldots,a_n,b_1,b_2)\in\BZ^{n+2}\simeq X_*(T)$ to $(a_1,\ldots,a_n,b_2,b_1)$. It follows that the group of coinvariants \begin{flalign*}
	  X_*(T)_\Gamma\simeq \BZ^n\times \frac{\BZ^2}{\BZ(1,-1)}\simeq \BZ^{n+1}
\end{flalign*}
is torsion-free. The coroot lattice $Q^\vee\sset X_*(T)$ consists of \begin{flalign*}
    \cbra{(a_1,\ldots,a_n,b_1,b_2)\in \BZ^{n+2}\ |\ a_1+\cdots+a_n+b_1 \text{\ is even, and\ } b_1+b_2=0 },
\end{flalign*} 
cf. \cite[\S 3.7]{smithling2011topological}. Hence, the coinvariants $(Q^\vee)_\Gamma\sset X_*(T)_\Gamma$ can be identified with \[\cbra{(r_1,\ldots,r_n,0)\in \BZ^{n+1}}. \] 
 The Kottwitz homomorphism for $T$ is \begin{equation} \label{kappaT}
   \begin{split}
   	  \kappa_T: T(K)\simeq (K\cross)^n\times L^\times &\twoheadrightarrow X_*(T)_\Gamma\simeq\BZ^{n+1} \\ (x_1,\ldots,x_n,y) &\mapsto (\val(x_1),\ldots,\val(x_n),\val_L(y))
   \end{split}.
\end{equation}
Denote by $\pi_1(G)$ the algebraic fundamental group of $G^\circ$. We have \begin{flalign*}
	\pi_1(G)_\Gamma=X_*(T)_\Gamma/(Q^\vee)_\Gamma\simeq \BZ.
\end{flalign*}

\begin{remark}
	Denote by $\breve K$ the completion of the maximal unramified extension of $K$. The Kottwitz homomorphism in \cite[\S 7.6]{kottwitz1997isocrystals} is stated in the form \begin{flalign*}
		  T(K)\twoheadrightarrow X_*(T)_{\Gal(\ol{K}/K)}^{\Gal(\breve K/K)}.  
	\end{flalign*}
	In our situation, $T$ splits over the totally ramified extension $L/K$. The above map is the same as $\kappa_T$.
\end{remark}

Let $\sB=\sB(G,K)$ denote the (extended) Bruhat--Tits building of $G$. The standard apartment $\CA$ associated with $S$ can be identified with \begin{flalign*}
	  X_*(S)\otimes_\BZ\BR \simeq \cbra{(r_1,\ldots,r_{2n},s)\in\BR^{2n+1}\ |\ r_1+r_{2n}=\cdots=r_n+r_{n+1}=s}.
\end{flalign*}
For $1\leq i\leq n$, let $\chi_i: S\ra \BG_{m,K}$ denote the character sending $\diag(x_1,\ldots,x_{2n},y,y)$ to $x_iy\inverse$. Denote by $\chi_{n+1}$ the character sending $\diag(x_1,\ldots,x_{2n},y,y)$ to $y$.  Then $(\chi_i)_{1\leq i\leq n+1}$ form a basis of the group $X^*(S)$ of characters.
Let $\Phi\coloneqq \Phi(G,S)\sset X^*(S)$ be the set of roots of $G$ relative to $S$. Then \begin{flalign*}
	\Phi=\cbra{\pm \chi_i\pm \chi_j\ |\ i,j\in [1,n] \text{\ and\ } i\neq j}\cup \cbra{\pm \chi_i\ |\ i\in[1,n]}
\end{flalign*}
is a root system of type $B_n$. In particular, the associated Weyl group $W(\Phi)$ is isomorphic to $S_{2n}^*$, the subgroup of $S_{2n}$ consisting of permutations $\sigma$ satisfying \begin{flalign*}
	\sigma(2n+1-i)+\sigma(i)=2n+1, \text{\ for all $1\leq i\leq 2n$}.
\end{flalign*} The affine root system is \begin{flalign*}
	\Phi_\aff = \cbra{\pm \chi_i\pm\chi_j+\BZ\ |\ i,j\in [1,n] \text{\ and\ }i\neq j }\cup \tcbra{\pm\chi_i+\half\BZ\ |\ i\in[1,n]};
\end{flalign*}
see \cite[Example 1.16]{tits1979reductive}.

We fix the base alcove $\fa$ in the apartment $X_*(S)\otimes\BR\sset \BR^{2n+1}$ determined by $$-1/2<\chi_1<\chi_2<\cdots<\chi_n<0.$$ Then $\fa$ has  $n+1$ vertices (minimal facets) $a_i+\BR\cdot(1,\ldots,1)$ for
\begin{flalign}  \label{vertices}
	a_i\coloneqq \rbra{(-1/2)^{(i)}, 0^{(2n-2i)},(1/2)^{(i)},0 }\in \BR^{2n+1},\ 0\leq i\leq n.
\end{flalign} 
The vertices corresponding to $a_0$ and $a_{n}$ are special; the other vertices are nonspecial.

Set \begin{flalign*}
	T(K)_1\coloneqq \ker\kappa_T\sset T(K).
\end{flalign*}
Define the Iwahori--Weyl groups $$\wt{W}\coloneqq N_{G}(K)/T(K)_1,\quad \wt{W}^\circ\coloneqq N_{G^{\circ}}(K)/T(K)_1, $$ where $N_{G}$ (resp. $N_{G^{\circ}}$) denotes the normalizer of $T$ in $G$ (resp. $G^{\circ}$). Then \[\wt{W}=\wt{W}^\circ\sqcup \tau\wt{W}^\circ. \]
The group $\wt{W}$ admits two semi-direct product decompositions\begin{flalign*}
    X_*(T)_\Gamma\rtimes W \text{\ and\ } W_{\aff}\rtimes \Omega,
\end{flalign*} 
where $W=N_G(K)/T(K)$ denotes the (finite) Weyl group, $W_{\aff}$ denotes the affine Weyl group, and $\Omega$ denotes the stabilizer subgroup of $\fa$. From now on, we identify \begin{flalign*}
	X_*(T)_\Gamma \simeq \cbra{(b_1,\ldots,b_n,c-b_n,\ldots,c-b_1,c)\in \BZ^{2n+1} }\sset \BZ^{2n+1}.
\end{flalign*}
   We also have \begin{itemize}
	\item $\wt{W}^\circ\simeq X_*(T)_\Gamma\rtimes W^\circ$, where $W^\circ=N_{G^\circ}(K)/T(K)\simeq W(\Phi)\simeq S_{2n}^*$. Note that $W/W^0\simeq\BZ/2\BZ$ is generated by $\tau$.
	\item $W_\aff \simeq (Q^\vee)_\Gamma\rtimes W^0 \simeq \BZ^{n}\rtimes S_{2n}^*$, where $(Q^\vee)_\Gamma$ denotes the coroot lattice \begin{flalign*}
		(Q^\vee)_\Gamma =\cbra{(b_1,\ldots,b_n,-b_n,\ldots,-b_1,0)\in \BZ^{2n+1}}\sset X_*(T)_\Gamma.	
	\end{flalign*}
	\item the Kottwitz homomorphism for $G^\circ$ is \begin{flalign*}
		  \kappa: G^\circ(K)\ra \pi_1(G)_\Gamma=X_*(T)_\Gamma/(Q^\vee)_\Gamma\simeq \BZ.
	\end{flalign*}
\end{itemize}

Note that the group $\wt{W}^\circ$ acts on $\BR^{2n+1}$ via affine transformations: for any vector $v\in \BR^{2n+1}$ and $w=t^ww_0$, where $w_0\in W^\circ\simeq S_{2n}^*$ and $t^w\in X_*(T)_\Gamma\sset \BZ^{2n+1}$, define $wv\in \BR^{2n+1}$ via setting 
\begin{flalign}
    &wv(i)\coloneqq v(w_0\inverse(i))+t^w(i), \text{\ for\ } 1\leq i\leq 2n+1.   \label{actionW}
\end{flalign}

\subsection{Parahoric subgroups}\label{ParahoricSubsection}
\begin{defn}

     Set \begin{flalign*}
     	\lambda_i &\coloneqq \CO_K\pair{t\inverse e_1,\ldots,t\inverse e_i,e_{i+1},\ldots,e_{2n},t\inverse f_1,f_2},\ 0\leq i\leq n.
     \end{flalign*}
     For $i\in[0,n]$, denote by $\lambda_{-i}=\lambda_i^\vee$ the dual lattice of $\lambda_i$ with respect to $\phi$.  We have \begin{flalign*}
     	   \lambda_{-i}\coloneqq \CO_K\pair{e_1,\ldots,e_{2n-i},te_{2n+1-i},\ldots,te_{2n},f_1,f_2 } \text{\ and\ }  \lambda_{-i}\sset \lambda_i\sset t\inverse\lambda_{-i}.
     \end{flalign*}
     We note 
     that $\lambda_0\neq \lambda_{-0}$. 
\end{defn}

\begin{defn} \label{defn-pmI}
	Fix the totally ordered set $$\CS\coloneqq \cbra{-n<-(n-1)<\cdots<-1<-0<0<1<2<\cdots<n},$$ where we distinguish the symbols $-0$ and $0$.   
	
	For any subset $I\sset [0,n]$, define \begin{flalign*}
		2n\BZ\pm I\coloneqq \cbra{2nd+i\ |\ d\in\BZ,i\in I}\sqcup \cbra{2nd-i\ |\ d\in \BZ, i\in I},
	\end{flalign*} 
	where the two symbols $2nd+0$ and $2nd-0$ are regarded as distinct elements. 
	We have an injection \begin{flalign*}
		\phi: 2n\BZ\pm I\hookrightarrow \BZ\times \CS,\qquad \phi(2nd\pm i)\coloneqq (d,\pm i).
	\end{flalign*}
	We equip $2n\BZ\pm I$ with the following total order: for $x,y\in 2n\BZ\pm I$, define $x<y$ if $\phi(x)<\phi(y)$ in the lexicographic order of $\BZ\times \CS$.
\end{defn}

For $\ell=2nd\pm i$ for some $d\in\BZ$ and $i\in [0,n]$, define $\lambda_\ell\coloneqq t^{-d}\lambda_{\pm i}$. For any non-empty subset $I\sset [0,n]$, we obtain a self-dual lattice chain \begin{flalign}
     	 \lambda_{I}\coloneqq \cbra{\lambda_\ell}_{\ell\in \cbra{2n\BZ\pm I}}.  \label{eq-lambdaI}
     \end{flalign}
     
 \begin{prop} \label{coro-stabvi}
 	The stabilizer subgroup of $a_j$ (see \eqref{vertices})  for $j\in [0,n]$ in $G(K)$ equals \begin{flalign*}
    	P_{\cbra{j}}\coloneqq \cbra{g\in G(K)\ |\ g\lambda_j=\lambda_j}.
    \end{flalign*}
 \end{prop}
 \begin{proof}
    This follows from the lattice-theoretic description of the Bruhat--Tits building $\sB$. Using \cite[Lemma 15.2.8]{kaletha2023bruhat}, a direct adaptation of the arguments in \cite[Corollary 2.6]{yang25TopFlat} yields the following: $\sB$ can be identified, $G(K)$-equivariantly, with the set of almost self-dual graded lattice chains of $V$, see \cite[\S 2]{yang25TopFlat} for terminology. 
    	The (self-dual) graded lattice chain corresponding to $a_j$ is given by \begin{flalign*}
		  (\lambda_{\cbra{j}}, c_{\cbra{j}}),
	\end{flalign*}
	where the grading $c_{\cbra{j}}$ is 
	\begin{flalign*}
		  c_{\cbra{j}}(\lambda_\ell) = \begin{cases}
		  	 -d -1/2 \quad &\text{if $\ell=2nd+j,d\in \BZ$,}\\ -d &\text{if $\ell=2nd-j,d\in\BZ$}.
		  \end{cases}
	\end{flalign*}
	Hence, the stabilizer subgroup of $a_j$ is the stabilizer of $(\lambda_{\cbra{j}},c_{\cbra{j}})$, which equals $P_{\cbra{j}}$.
 \end{proof}

\begin{defn}
	For a non-empty subset $I\sset [0,n]$, set \begin{flalign}
		P_I\coloneqq \cbra{g\in G(K)\ |\ g\lambda_i=\lambda_i, i\in I} \text{\ and\ } Q_I\coloneqq P_I\cap G^\circ(K) \label{PJsubgroup}
	\end{flalign}
	If $I=\cbra{i}$ is a singleton, we simply write $P_i$ for $P_{\cbra{i}}$.
\end{defn}

\begin{defn}
	We say $P\sset G(K)$ is a \dfn{parahoric subgroup of $G(K)$} if $P$ is a parahoric subgroup of $G^\circ(K)$ in the usual sense of Bruhat--Tits theory (for connected reductive groups).  
\end{defn}

\begin{thm} \label{prop-paraIndex}
   The subgroup $Q_I$ in \eqref{PJsubgroup} is a parahoric subgroup of $G(K)$. Any parahoric subgroup of $G(K)$ is $G^\circ(K)$-conjugate to $Q_I$ for a unique non-empty subset $I\sset [0,n]$ with the property that if $i\in I$ with $i\geq \lfloor n/2\rfloor$, then $n-i\in I$.
   
   In particular, the $G^\circ(K)$-conjugacy classes of maximal parahoric subgroups of $G^\circ(K)$ are in bijection with the set $\cbra{Q_{i}\ |\ 0\leq i\leq \lfloor n/2\rfloor }$.
\end{thm}

\begin{proof} 
    (cf. \cite[Proposition 2.9]{yang25TopFlat}) Note that any non-empty $I\sset [0,n]$ determines a facet $\ff_I$ in the base alcove $\fa$, where the vertices of the closure $\ol{\ff}_I$ are given by $$\cbra{a_i+\BR(1,\ldots,1)\ |\ i\in I },$$ notation as in \eqref{vertices}. 
	 By Proposition \ref{coro-stabvi}, the (pointwise) stabilizer subgroup of $\ff_I$ is equal to $P_I$.
	 Since the target of the Kottwitz homomorphism $\kappa: G^\circ(K)\ra \pi_1(G^\circ)_\Gamma\simeq\BZ$ is torsion-free, we have \[Q_I=P_I\cap\ker\kappa,\] 
	and $Q_I$ is a parahoric subgroup of $G(K)$ by \cite[Proposition 3]{haines2008parahoric}.
	
	Note that any facet in the building $\sB$ is $G^\circ(K)$-conjugate to a facet $\ff$ in the alcove $\fa$. Set $I\coloneqq \cbra{i\in [0,n]\ |\ a_i\in \ov{\ff}}$. Then the parahoric subgroup attached to $\ff$ is equal to $Q_I$. 
	
	Let $G_{\mathrm{sc}}$ denote the simply connected cover of the derived subgroup $G_\der$. Let $T_{\mathrm{sc}}$ be the maximal torus of $G_\sc$.    Let $H_0$ denote the image of $T_{\mathrm{sc}}(K)$ in $T(K)$. Then \begin{flalign*}
		 H_0=(\ker\kappa)\cap T_\der(K)=T_\der(K).
	\end{flalign*} 
	Let $H_1$ denote the subgroup of $T(K)$ generated by \begin{flalign*}
		t_1\coloneqq \diag(1^{(n)},t^{(n)},\begin{psmallmatrix}
			0 &1\\ -t &0
		\end{psmallmatrix}). 
	\end{flalign*} 
	Note that the Kottwitz homomorphism \begin{flalign}
		\kappa_T: T(K)\twoheadrightarrow X_*(T)_\Gamma\simeq\BZ^{n+1}  \label{eq-kappaT}
	\end{flalign}  takes $H_0T(K)_1$ to the coroot lattice $(Q^\vee)_\Gamma\sset (X_*(T))_\Gamma$.
	Since the image of $H_0H_1\sset T(K)$ under \eqref{eq-kappaT} generates $X_*(T)_\Gamma$,  we have \begin{flalign*}
		  T(K)=T(K)_1H_0H_1.
	\end{flalign*}
	Denote by $\Delta$ the local Dynkin diagram of $G$. By \cite[\S 2.5]{tits1979reductive},  the $G^\circ(K)$-action \begin{flalign}
		G^\circ(K) \ra \Aut(\Delta)  \label{GDelta}
	\end{flalign}
	factors through $\kappa: G^\circ(K)\ra \pi_1(G^\circ)_\Gamma$, and the image in $\Aut(\Delta)$ 
	is isomorphic to $$T(K)/(T(K)_1C(K)H_0),$$ where $C(K)$ denotes the group of $K$-points of the center $C$ of $G$. It is straightforward to check that \begin{flalign*}
		\Xi \simeq \frac{T(K)_1H_0H_1}{T(K)_1H_0C(K)}\simeq \BZ/2\BZ
	\end{flalign*}
	generated by $t_1$. 
	Set \begin{flalign*}
		\tau_1 &\coloneqq \begin{cases}
			t_1\sigma_1\diag(1^{(2n)},-1,1) \in N_{G^\circ}(K) \quad &\text{if $n$ is odd};\\ t_1\sigma_1\in N_{G^\circ}(K) &\text{if $n$ is even}.
		\end{cases} 
	\end{flalign*}
	Here, $\sigma_1\in S_{2n}^*$ is the permutation sending $(x_1,\ldots,x_{2n})$ to $(x_{n+1},\ldots,x_{2n},x_1,\ldots,x_n)$. Then $\tau_1\in \Omega$ stabilizes the base alcove $\fa$, and $\kappa(\tau_1)=\kappa(t_1)$. Hence, $\Xi$ is also generated by (the action of) $\tau_1$. Let us identify the vertices of $\Delta$ with the set $\tcbra{\lambda_{\cbra{i}}}_{i\in [0,n]}$ of self-dual lattice chains $\lambda_{\cbra{i}}$ indexed by $[0,n]$. The $\Xi$-action on $\Delta$ can be explicitly described in the following way: $\tau_1\lambda_i=\lambda_{-(n-i)}$ for $0\leq i\leq n$. 
	
	Note that for lattices $\Lambda,\Lambda'$ in $V$ and $g\in G^\circ(K)$, if $g\Lambda=\Lambda$, then $g\Lambda^\vee=\Lambda^\vee$; and if $g\Lambda=\Lambda'$, then $g\Stab(\Lambda)g\inverse=\Stab(\Lambda')$, where $\Stab(\Lambda)=\cbra{h\in G(K)\ |\ h\Lambda=\Lambda}$.  
	By \cite[\S 2.5]{tits1979reductive}, the $G^\circ(K)$-conjugacy classes of parahoric subgroups are in bijection with the $\Xi$-orbits on the set of non-empty subsets of $\Delta$. By the above $\Xi$-action, this is in bijection with the set of non-empty subsets $I$ of $[0,n]$ with the property that if $i\in I$ with $i\geq \lfloor n/2\rfloor$, then $n-i\in I$.  
\end{proof}

\begin{remark}\label{in-i}
	By the proof of Theorem \ref{prop-paraIndex}, the parahoric subgroups $Q_i$ and $Q_{n-i}$ are conjugate by $\tau_1\in G^\circ(K)$. 
\end{remark}

An immediate corollary of Theorem \ref{prop-paraIndex} is the following.
\begin{corollary}
	Any self-dual lattice chain $\CL$ of $V$ is $G^\circ(K)$-conjugate to $\lambda_I$ (see \eqref{eq-lambdaI}) for a unique subset $I\sset [0,n]$ with the property that if $i\in I$ with $i\geq \lfloor n/2\rfloor$, then $n-i\in I$. 
\end{corollary}

\begin{remark}
	In the literature (e.g. \cite{I.Z,PaZa,howard2017pappas}), authors also use vertex lattices to describe maximal parahoric subgroups. In our setting, the standard lattice $\Lambda_i\sset V$ gives a vertex lattice such that $\pi_0\Lambda_i\subset \Lambda_i^\vee\sset \Lambda_i$ in $V$ and $\dim_k(\Lambda_i/\Lambda_i^\vee)=2i+1$.  
\end{remark}

\section{Naive and spin local models} \label{sec-naivespinmod}
In this section, we first review the constructions of naive and spin local models following \cite{RZbook,PR,smithling2011topological,Yang25}. We then identify their special fibers with unions of Schubert cells in an affine flag variety. These Schubert cells are indexed by certain permissible subsets in the Iwahori--Weyl groups.

Recall from \S\ref{MainResults} the quadratic extension $F/F_0$, and let $G\coloneqq \GO(V,\phi)$ be the orthogonal similitude group attached to a symmetric space $(V,\phi)$ of $F_0$-dimension $2n+2$ with distinguished $F_0$-basis $e_1,\ldots,e_{2n},f_1,f_2$. Fix a non-empty subset $I\sset [0,n]$.  We will denote the standard lattice chain $\lambda_I$ in  \eqref{eq-lambdaI} by $\Lambda_I$. 

\begin{defn}[{cf. \cite{RZbook,smithling2011topological}}] \label{defn-naive}
    The \dfn{naive local model} $\RM^\naive_I$ is the projective scheme over $\CO_{F_0}$ representing the functor that sends each $\CO_{F_0}$-algebra $R$ to the set of $R$-modules $(\CF_\Lambda)_{\Lambda\in \Lambda_I}$ such that 
    \begin{itemize}
        \item [LM1.] for any $\Lambda\in \Lambda_I$, Zariski locally on $\Spec R$, $\CF_\Lambda$ is a direct summand of $\Lambda_R\coloneqq \Lambda\otimes_{\CO_{F_0}} R$ of rank $n+1$;
        \item[LM2.] for any $\Lambda\in \Lambda_I$, the perfect pairing \[\phi\otimes 1: \Lambda_R\times \Lambda^\vee_R \ra R \] induced by $\phi$ satisfies $(\phi\otimes 1)(\CF_\Lambda, \CF_{\Lambda^\vee})=0$;
        \item[LM3.] for any inclusion $\Lambda\sset \Lambda'$ in $\Lambda_I$, the natural map $\Lambda_R\ra \Lambda'_R$ induced by $\Lambda\hookrightarrow \Lambda'$ sends $\CF_\Lambda$ to $\CF_{\Lambda'}$; and the isomorphism $\Lambda_R\simto (\pi_0\Lambda)_R$ induced by $\Lambda\overset{\pi_0}{\ra}\pi_0\Lambda$ identifies $\CF_\Lambda$ with $\CF_{\pi_0\Lambda}$.
    \end{itemize}
\end{defn}
We have \begin{equation*}
    \RM^\naive_{I,F_0} \simeq \OGr(n+1,V), 
\end{equation*} where $\OGr(n+1,V)$ denotes the orthogonal Grassmannian of (maximal) totally isotropic $(n+1)$-dimensional subspaces of $V$. By the standard theory of orthogonal Grassmannian (see e.g. \cite[\S 85]{elman2008algebraic}), we have the following.

\begin{lemma}
	The generic fiber $\RM^\naive_{I,F_0}$ is connected of dimension $n(n+1)/2$. The base change $\RM^\naive_{I,F_0}\otimes_{F_0}F$ consists of two connected components $\OGr(n+1,V)_\pm$, which are isomorphic to each other. 
\end{lemma}

We can identify $\OGr(n+1,V)_\pm$ with the flag variety $G^\circ_F/P_{\mu_\pm}$, where $P_{\mu_\pm}$ denotes the parabolic subgroup (over $F$) associated to a minuscule cocharacter $\mu_\pm$ given by \begin{flalign}
    \mu_+\coloneqq (1^{(n+1)},0^{(n+1)}) \text{\ and\ } \mu_-\coloneqq (1^{(n)},0,1,0^{(n)}). \label{cochar}
\end{flalign}

\begin{conj}[{\cite[Conjecture 2.36]{PRS}}] \label{conj-topoflat}
	The naive local model $\RM^\naive_I$ is topologically flat.
\end{conj}

We prove Conjecture~\ref{conj-topoflat} in \S\ref{sec-topoflat}. Following \cite[\S 7]{PR}, Pappas and Rapoport defined an involution operator $a$ on $\wedge^{n+1}_FV_F$ inducing a decomposition  \begin{flalign}
	\wedge^{n+1}_FV_F= W_+\oplus W_-, \label{Wdecomp}
\end{flalign} where $W_\pm$ denotes the $\pm 1$-eigenspace for $a$. For an $\CO_F$-lattice $\Lambda$ in $V_F$, set \[(\wedge^{n+1}_{\CO_F}\Lambda)_\pm \coloneqq \wedge^{n+1}_{\CO_F}\Lambda\cap W_\pm.  \]
We now define a refinement of $\RM^\naive_I$ as follows; see \cite[Definition 3.2]{yang25TopFlat}.

\begin{defn}\label{defn-spin}
    The spin local model $\RM^\pm_I$ is the projective scheme over $\CO_F$ representing the functor sending each $\CO_F$-algebra $R$ to the set of $R$-modules $(\CF_\Lambda)_{\Lambda\in \Lambda_I}$ such that $(\CF_\Lambda)_{\Lambda\in \Lambda_I}\in \RM^\naive_I(R)$ and 
    \begin{itemize}
        \item[LM4$\pm$.] for any $\Lambda\in\Lambda_I$, Zariski locally on $\Spec R$,  the line $\bigwedge_{R}^{n+1} \mathcal{F}_{\Lambda}$ is contained in
    \[
    \operatorname{Im}\left[ \left( \wedge_{\CO_F}^{n+1} \Lambda_{\CO_F} \right)_{\pm} \otimes_{\CO_F} R \longrightarrow \wedge_{R}^{n+1} \Lambda_{R} \right].
    \]
    \end{itemize}
\end{defn}

The generic fiber of $\RM^\pm_I$ is isomorphic to $\OGr(n+1,V)_\pm$, cf. \cite[8.2.1]{PR}). 
Let \begin{flalign}
	\sG\coloneqq \sG_{I}  \label{Gpara}
\end{flalign}  
the (affine smooth) group scheme of similitude automorphisms of $\Lambda_I$.  By Bruhat--Tits theory, the neutral component $\sG^\circ$ is a parahoric group scheme for $G^\circ$. Then $\sG$ naturally acts on $\RM^\naive_I$, and  $\sG^\circ$ preserves $\RM^\pm_I$, cf. \cite[Lemma 3.4]{yang25TopFlat}.

\begin{defn}\label{Mlocdefin}
    Denote by $\RM^{\pm\loc}_I$ the schematic closure of $\RM^\pm_{I,F}$ in $\RM^\pm_I$.
\end{defn}
 

Let $\RM^{PZ}_{\sG^\circ,\mu_\pm}$ denote the {Pappas--Zhu local model} associated with $\sG^\circ$ and the minuscule cocharacter $\mu_\pm$ in \eqref{cochar}; cf. \cite{PZ}. As explained in \cite[Proposition 3.6]{yang25TopFlat}, we have the following. 

\begin{prop}\label{propiso}
	The scheme $\RM^{\pm\loc}_I$ is isomorphic to $\RM^{PZ}_{\sG^\circ,\mu_\pm}$ and represents the corresponding v-sheaf local model in the sense of Scholze--Weinstein \cite[\S 21.4]{scholze2020berkeley}. In particular, $\RM^{\pm\loc}_I$ is $\CO_F$-flat of relative dimension $n(n+1)/2$, normal, and Cohen-Macaulay with reduced special fiber. 
\end{prop}

\begin{conj}[{\cite[Conjecture 8.1]{PR}}] \label{conj-pr}
    The spin local model $\RM^\pm_I$ is flat over $\CO_F$. Equivalently, $\RM^\pm_I=\RM^{\pm\loc}_I$.
    \end{conj}
We prove Conjecture~\ref{conj-pr} in \S\ref{FlatSec}. In the rest of the paper, we let $\RM^\naive_{I,k}$ (resp. $\RM^\pm_{I,k}$) denote the special fiber of $\RM^\naive_{\Lambda_I}$ (resp. $\RM^\pm_{\Lambda_I}$). Sometimes, we simply write $\RM^\naive_k$ or $\RM^\pm_k$ when the index set is clear.

\section{Topological flatness of naive local models} \label{sec-topoflat}
In this section, we prove Theorem \ref{intro-thmtopo}, namely that the naive local model is topologically flat over $\CO_{F_0}$ for any parahoric level structure. We first treat the maximal parahoric case in \S \ref{subsec-pseud}, establishing Theorem \ref{coro-topoflat}. In \S \ref{subsec-genepara}, we then deduce topological flatness for an arbitrary parahoric level by reducing to the maximal case, using the vertexwise criterion \cite{haines2017vertexwise} for admissible subsets.

\subsection{Embedding \texorpdfstring{$\RM^\naive_{I,k}$}{M\^{}naive\_{}I,k} in the affine flag variety}
From now on, we take the field $K$ in \S \ref{sec-para} to be the field $k((t))$ of Laurent series over $k$. For ease of notation, we use the same symbol $G$ for the orthogonal similitude group over $F_0$ or $K$; the base field will be clear from context.

Let $I$ be a non-empty subset of $[0,n]$.
For the associated self-dual lattice chain $\lambda_I$ (see \eqref{eq-lambdaI}) of $K^{2n+2}$, we let $\CP_I$ be the $\CO_K$-group scheme of similitude automorphisms of $\lambda_I$. As $p\neq 2$, the group scheme $\CP_I$ is affine and smooth over $\CO_K$ by \cite[Theorem 3.16]{RZbook}, and the generic fiber of $\CP_I$ is $G$. The neutral component $\CP^{\circ}_I$ of $\CP_I$ is the parahoric group scheme attached to $\lambda_I$. Denote by $\CQ_I$ the schematic closure of $G^{\circ}$ in $\CP_I$. Then $\CQ_I$ is a connected component of $\CP_I$. Set $Q_I\coloneqq \CQ_I(\CO_K)$. We have \begin{flalign*}
	P_I=\CP_I(\CO_K), Q_I=P_I\cap G^\circ(K), \text{\ and\ } P^0_I=\CP^{\circ}_I(\CO_K).
\end{flalign*} 
\begin{lemma} \label{lem-Itype}
	Let $I\sset [0,n]$ be a non-empty subset. Then $P_I/Q_I\simeq \BZ/2\BZ$ is generated by $\tau= \diag(1^{(2n)},-1,1)$, and $Q_I=P_I^\circ$.
\end{lemma}
\begin{proof}
It is clear that the element $\diag(1^{(2n)},-1,1)\in P_I-Q_I$. Since \begin{flalign*}
	P_I/Q_I\hookrightarrow G(K)/G^\circ(K)=\BZ/2\BZ,
\end{flalign*}
we have $P_I/Q_I\simeq \BZ/2\BZ$. As $\pi_1(G^\circ)_\Gamma\simeq\BZ$ is torsion-free, we obtain that $Q_I=P_I^\circ$.
\end{proof}
In particular, $\CQ_I=\CP_I^\circ$, so $\CQ_I$ is parahoric.
\begin{defn}[{\cite[Definition 3.9]{yang25TopFlat}}]  \label{defnFlag}
	The affine flag variety $\sFl_I$ is the fpqc quotient sheaf\footnote{Although $\CP_I$ need not equal $\CP_I^\circ$ in general, we still refer to $\sFl_I$ as an affine flag variety. } $$\sFl_I\coloneqq LG/L^+\CP_I.$$ Taking $\lambda_I$ as the base point, one identifies $\sFl_I$ with the fpqc sheaf sending a $k$-algebra $R$ to the set of $R[[t]]$-lattice chains $(L_{\ell})_{\ell\in 2n\BZ\pm I}$ in $R((t))^{2n+2}$ with the property that \begin{enumerate}
	\item for $\ell<\ell'$ in $2n\BZ\pm I$, we have $L_\ell\sset L_{\ell'}$, and the quotient $L_{\ell'}/L_\ell$ is locally free of $R$-rank $\dim_k\lambda_{\ell'}/\lambda_\ell$;
	\item if $\ell'=\ell+2nd$ for some $d\in\BZ$, then $L_{\ell'}=t^{-d}L_\ell$;
	\item Zariski-locally on $\Spec R$, there exists a scalar $\alpha\in R((t))\cross$ such that ${L}_\ell^{\phi}=\alpha L_{-\ell}$ for all $\ell\in 2n\BZ\pm I$. Here ${L}_\ell^{\phi}\coloneqq \cbra{x\in R((t))^{2n+2} }\ |\ \phi(L_\ell,x)\sset R[[t]] . $
\end{enumerate} 
\end{defn}

Note that $G=G^\circ\sqcup\tau G^\circ$.
By Lemma \ref{lem-Itype}, we have  \begin{flalign} \label{sFI}
    \sFl_I= LG/L^+\CP_I \simeq LG^{\circ}/L^+\CP_I^\circ. 
\end{flalign}

\begin{defn}
    For $w\in\wt{W}$, the Schubert cell $C_w$ is the reduced $k$-subscheme  \[C_w\coloneqq L^+\CP^{\circ}_Iw \sset\sFl_I, \]
    and the Schubert variety $S_w$ is the reduced closure of $C_w$ in $\sFl_I$. 
\end{defn}

\begin{prop} \label{lem-cell}
	There is a bijection between the set of Schubert cells in $\sFl_I$ and the double coset $W_I\backslash \wt{W}^\circ/W_I$,
	where $W_I\coloneqq \cbra{w\in W_\aff\ |\ wa_i=a_i \text{\ for\ } i\in I}$.
\end{prop}
\begin{proof}
   By \eqref{sFI}, the lemma follows from the standard affine Bruhat decomposition for connected reductive groups (see  \cite[Proposition 4.8, pp. 184]{PRS}).
\end{proof}

\begin{lemma}\label{lem-Wmperm}
    Suppose that $I=\cbra{i}\sset [0,n]$. Denote $W_i\coloneqq W_I$.
    We have $$W_i\simeq S_{2i}^*\times S_{2(n-i)}^*\sset S_{2n}^*,$$ where the two factors permute $A_i\coloneqq \sbra{1,i}\cup [i^*,2n]$ and $B_i\coloneqq \sbra{i+1,i^*-1}$ respectively.
\end{lemma}
\begin{proof}(cf. \cite[Lemma 3.14]{yang25TopFlat})
    Let $w=t^ww_0\in W_i\sset W_\aff$. By definition, we have $wa_i=a_i$. Namely,  \begin{flalign}
    	a_i(w_0\inverse(j))+t^w(j)=a_i(j) \text{\ for $1\leq j\leq 2n$, and $t^w(2n+1)=0$}.  \label{am}
    \end{flalign}
    Hence, the translation part \begin{equation}
    	t^w=w_0a_i-a_i \label{tw}
    \end{equation} is determined by $w_0$. 
    Since $w_0\in S_{2n}^*$ and $w_0a_i-a_i=t^w\in\BZ^{2n+1}$, we conclude that $w_0$ permutes the subsets $A_i$ and $B_i$.  
   Therefore, the projection $\wt{W}^\circ\twoheadrightarrow S_{2n}^*$ induces a homomorphism $$f:W_i\ra  S_{2i}^*\times S_{2(n-i)}^*. $$
    By \eqref{tw}, the map $f$ is an injection. For any $w_0\in S_{2i}^*\times S_{2(n-i)}^*$, the equation \eqref{tw} defines an element $t^ww_0$ in $W_i$. It follows that $f$ is also surjective; hence $W_i\simeq S_{2i}^*\times S_{2(n-i)}^*$. 
\end{proof}

Using the obvious isomorphism $\CO_{F_0}/\pi_0\CO_{F_0}\simeq k\simeq \CO_K/t\CO_K,$ we may identify \[\Lambda_I\otimes_{\CO_{F_0}} k\simeq \lambda_I\otimes_{\CO_K}k. \] 
Let $R$ be a $k$-algebra and let $(\CF_{\Lambda_i})_{\Lambda_i\in\Lambda_I}\in \RM^\naive_I(R)$ (see \S \ref{sec-naivespinmod}). Denote by  $$\wt{\CF}_{\Lambda_i}\sset \lambda_i\otimes_{\CO_K}R[[t]] $$ the inverse image of the $R$-submodule $\CF_{\Lambda_i}\sset \Lambda_i\otimes_{\CO_{F_0}} R\simeq\lambda_i\otimes_kR$ along the natural reduction map $\lambda_i\otimes_{\CO_K}R[[t]]\twoheadrightarrow \lambda_i\otimes_kR$. As in \cite[\S 3.2]{yang25TopFlat}, we obtain a closed immersion \begin{flalign*}
    \RM^\naive_{I,k}\hookrightarrow \sFl_I,\quad (\CF_{\Lambda_i})\mapsto (\wt{\CF}_{\Lambda_i}).
\end{flalign*}
Furthermore, the underlying topological spaces of $\RM^\naive_{I,k}$ and $\RM^\pm_{I,k}$ are unions of Schubert cells in $\sFl_I$.

\begin{defn}\label{defn-permi}
    Fix a non-empty subset $I\sset [0,n]$. 
    For $w\in\wt{W}^\circ$, we say $w$ is \dfn{naively-permissible} (with respect to $I$) if the Schubert cell $C_w$ is contained in $\RM^\naive_{I,k}$.
\end{defn}

\subsection{Maximal parahoric case} \label{subsec-pseud}
Suppose $I=\cbra{i}$, where $0\leq i\leq n$. 

Let $(\alpha_j)_{1\leq j\leq 2n+2}$ and $(\beta_j)_{1\leq j\leq 2n+2}$ denote the following ordered $\CO_{F_0}$-bases of $\Lambda_i$ and $\Lambda_{-i}$, respectively:
\begin{flalign*}
    \Lambda_i &: \cbra{\pi_0\inverse e_1,\ldots,\pi_0\inverse e_i,e_{i+1},\ldots,e_{2n},\pi_0\inverse f_1,f_2 } \text{and\ }\\ \Lambda_{-i} &:\cbra{e_1,\ldots,e_{2n-i},\pi_0 e_{i^*},\ldots,\pi_0 e_{2n},f_1,f_2}.
\end{flalign*} 
For an $\CO_{F_0}$-algebra $R$, we use the same notation for the induced $R$-bases of $\Lambda_i\otimes R$ and $\Lambda_{-i}\otimes R$.

\begin{defn}\label{defn-face}
	An $I$-face (or simply \dfn{face}) is a pair $(v_i,v_{-i})$ of vectors in $\BZ^{2n+1}$ such that \begin{enumerate}
        \item $v_{-i}\geq v_i\geq v_{-i}-\mathbf 1 $; 
        \item $\Sigma v_i=\Sigma v_{-i}-2i-1$;
        \item there exists $d\in\BZ$ such that $v_i^*+v_{-i}=\mathbf d_0$ (notation as in \S \ref{subsec-notation-intro}).
    \end{enumerate}  
\end{defn}

Note that the action of $\wt{W}^\circ$ on $\BR^{2n+1}$ (see \eqref{actionW}) induces an action on the set of faces. Set \begin{flalign*}
	\omega_{i}\coloneqq ((-1)^{(i)},0^{(2n-i)},-1)\in \BZ^{2n+1} \text{\ and\ } \omega_{-i}\coloneqq (0^{(2n-i)},1^{(i)},0)\in\BZ^{2n+1}.
\end{flalign*}
The pair $(\omega_i,\omega_{-i})$ is clearly a face.

Let $R$ be a $k$-algebra. By construction, the image of the embedding $\RM^\naive_{\cbra{i},k}(R)\hookrightarrow \sFl_{\cbra{i}}(R)$ consists of $(L_i,L_{-i})$ in $\sFl_{\cbra{i}}(R)$ such that, for $j=\pm i$, 
\begin{flalign}\label{m1m2}
\begin{aligned}
&\text{(M1)}\quad \lambda_{j}\otimes_{\mathcal{O}_K} R[[t]]\supset L_j\supset t\lambda_{j}\otimes_{\mathcal{O}_K}R[[t]], \\
&\text{(M2)}\quad (\lambda_j\otimes_{\mathcal{O}_K}R[[t]])/L_j \text{ is locally free of rank } n+1.
\end{aligned}
& & 
\end{flalign}

\begin{lemma}\label{permp1p2}
	Let $w\in\wt{W}^\circ$. Then $w$ is naively-permissible (see Definition \ref{defn-permi}) if and only if  \begin{enumerate}
    \item[(P1)] for $j=\pm i$, we have $\omega_j\leq w\omega_j\leq \omega_j+\mathbf 1$;
    \item[(P2)] $\Sigma (w\omega_i)=n-i$.
    \end{enumerate}
    Moreover, in this case, we have $w\omega_i(2n+1)=0$ and $w\omega_{-i}(2n+1)=1$.
\end{lemma}
\begin{proof}
	(cf. \cite[Lemma 3.18]{yang25TopFlat}) Note that $w$ corresponds to $(w\lambda_i,\omega\lambda_{-i})\in\sFl_{\cbra{i}}(k)$. If $w\omega_{\pm i}=(r_1,\ldots,r_{2n},s)\in\BZ^{2n+1}$, then a direct computation gives  \begin{flalign}
		  w\lambda_{\pm i}=k[[t]]\pair{t^{r_1}e_1,\ldots,t^{r_{2n}}e_{2n},t^{\lfloor s/2\rfloor}f_1, t^{\lceil s/2\rceil}f_2 }. \label{ai}
	\end{flalign} 
	Note that $w$ is naively-permissible if and only if $w\lambda_{\pm i}$ satisfies (M1) and (M2). By \eqref{ai}, we obtain that (M1) (resp. (M2)) is equivalent to (P1) (resp. (P2)). 
	
	Let $w\in\wt{W}^\circ$ be naively permissible. Set $\CF_i\coloneqq w\lambda_i/t\lambda_i$ and $\CF_{-i}\coloneqq w\lambda_{-i}/t\lambda_{-i}$. Then $(\CF_i,\CF_{-i})\in \RM^\naive_{\cbra{i}}(k)$. By (P1),  we have that $w\omega_i(2n+1)$ equals $-1$ or $0$. Suppose $w\omega_i(2n+1)=-1$. Then $w\omega_{-i}(2n+1)=0$, and hence, $k\pair{\alpha_{2n+1},\alpha_{2n+2}}\sset \CF_i$ and $k\pair{\beta_{2n+1},\beta_{2n+2}}\sset \CF_{-i}$. This contradicts the fact that $\CF_i$ and $\CF_{-i}$ are orthogonal to each other. We obtain that $w\omega_i(2n+1)=0$ and $w\omega_{-i}(2n+1)=1$. 
\end{proof}


\begin{lemma} \label{lem-perm}
    Let $w\in\wt{W}^\circ$ be naively-permissible.  
    \begin{enumerate}
        \item Set $v_{\pm i}\coloneqq w\omega_{\pm i}$. We have $v_i^*+v_{-i}=(1^{(2n)},0)$. In particular, $v_{-i}$ is completely determined by $v_i$. 
        \item Set $\mu^w_{\pm i}\coloneqq w\omega_{\pm i}-\omega_{\pm i}$. The point $(w\lambda_{i},w\lambda_{-i})\in \sFl_i(k)$ corresponds to $(\CF^w_i,\CF^w_{-i})\in \RM^\naive_i(k)$, where \begin{flalign*}
        \CF^w_i&=k\pair{\alpha_j,\alpha_{2n+2}\ |\ \mu_i^w(j)=0 \text{\ for\ } 1\leq j\leq 2n} \text{\ and\ }\\ \CF^w_{-i} &=k\pair{\beta_j,\beta_{2n+1}\ |\ \mu_{-i}^w(j)=0 \text{\ for\ } 1\leq j\leq 2n  }.
    \end{flalign*}
    \end{enumerate} 
\end{lemma}
\begin{proof}
    (1) By Lemma \ref{permp1p2}, we have $v_i(2n+1)=0$ and $v_{-i}(2n+1)=1$. It follows that $(v_i^*+v_{-i})(2n+1)=0$. Since $(v_i,v_{-i})$ is a face, there exists a constant $d\in\BZ$ such that $$d=v_i(j)+v_{-i}(j^*)$$ for all $1\leq j\leq 2n$. By condition (P1), we have  \begin{flalign*}
        \omega_{i}(j)+\omega_{-i}(j^*)\leq v_i(j)+v_{-i}(j^*)\leq \omega_{i}(j)+\omega_{-i}(j^*)+2.
    \end{flalign*}
    Since $\omega_{-i}(j)+\omega_{i}(j^*)=0$ for $1\leq j\leq 2n$, we obtain that $d\in\cbra{0,1,2}$. Suppose $d=0$. Then $v_i(j)$ must attain its minimal possible value, namely $v_i=\omega_{i}$. However, this contradicts condition (P2). Similarly, $d=2$ is also impossible. We conclude that $d=1$. Therefore, $v_i^*+v_{-i}=(1^{(2n)},0)$. 

    (2) By \eqref{ai} in the proof of Lemma \ref{permp1p2}, the $k[[t]]$-lattice $w\lambda_i$ is \[k[[t]]\pair{t^{r_1}e_1,\ldots, t^{r_{2n}}e_{2n},f_1, f_2}, \]
        where $(r_1,\ldots,r_{2n},0)=w\omega_i\in\BZ^{2n+1}$. Then we obtain that $$\CF^w_i=w\lambda_i/t\lambda_i=  k\pair{\alpha_j,\alpha_{2n+2}\ |\ \mu_i^w(j) =0\text{\ for\ } 1\leq j\leq 2n}.$$
        We note that as $w$ is naively-permissible, $\CF^w_i$ is indeed an $n+1$-dimensional $k$-vector space. Similarly, we have \begin{flalign*}
        	\CF^w_{-i} &=k\pair{\beta_j,\beta_{2n+1}\ |\ \mu_{-i}^w(j)=0 \text{\ for\ } 1\leq j\leq 2n  }.
        \end{flalign*}
\end{proof}

\begin{lemma} \label{lem-Wmfaces}
	The map $w\mapsto w(\omega_i,\omega_{-i})$ induces a bijection \begin{flalign*} 
		 \wt{W}^\circ/W_i \simto \cbra{\text{faces}}.
	\end{flalign*} 
    Recall that $W_i=W_{\cbra{i}}$ is defined in Proposition \ref{lem-cell}.
\end{lemma}
\begin{proof}
	Using Lemma \ref{lem-Wmperm}, the proof proceeds similarly as in \cite[Lemma 4.2]{yang25TopFlat}. We omit the details.
\end{proof}

\begin{defn}
    Let $E$ be a subset of $\sbra{1,2n}$ with cardinality $n$. \begin{enumerate}
        \item We say that a subspace $\CF_i\sset \Lambda_i\otimes k$ (resp. $\CF_{-i}\sset \Lambda_{-i}\otimes k$)  is \dfn{given by} $E$ if $\CF_i$ has a $k$-basis consisting of $f_j$ (resp. $g_j$) for $j\in E$ and $f_{2n+2}$ (resp. $g_{2n+1}$). 
        \item We say that $E$ is \dfn{naively-permissible} (with respect to $\cbra{i}$) if, for every pair $\cbra{j,j^*}\sset A_i$ (notation as in Lemma \ref{lem-Wmperm}), the pair is not contained in $E$, and for every pair $\cbra{j,j^*}\sset B_i$, at least one element lies in $E$.
    \end{enumerate} 
\end{defn}

Let $w\in \wt{W}^\circ$ be naively-permissible. A similar argument as \cite[Lemma 4.5]{yang25TopFlat} shows that \[E^w\coloneqq \cbra{j\ |\ \mu_i^w(j)=0\text{\ and\ } 1\leq j\leq 2n} \] is naively-permissible. 
By Lemma \ref{lem-perm}, the subspace $\CF^w_i$ (resp. $\CF^w_{-i}$) is given by $E^w$ (resp. $(E^w)^\perp$). Here, $(E^w)^\perp$ denotes the complement in $[1,2n]$ of the subset $\cbra{j^*\ |\ j\in E^w}$.
 
 Clearly, for $u\in W_i$, we have $$E^w=E^{wu}.$$ 


\begin{lemma}\label{lem-orbits4}
    There are precisely $\min{\cbra{i,n-i}}+1$ orbits of naively-permissible subsets under the action of $W_i=S_{2i}^*\times S_{2n-2i}^*$.
\end{lemma}
\begin{proof}
    Let $E$ be a naively-permissible subset. Let $a_1$ (resp. $a_0$) denote the number of pairs in $A_i$ contributing one (resp. zero) element in $E$. Let $b_2$ (resp. $b_1$) denote the number of pairs in $B_i$ contributing two (resp. one) elements in $E$. Since $E$ is naively-permissible, we have \begin{flalign*}
        \max{\cbra{0,2i-n}}\leq a_1\leq i,\ a_0=i-a_1,\ b_1=n-2i+a_1,\ b_2=i-a_1.
    \end{flalign*}
    Note that the lower bound for $a_1$ is required to guarantee $b_1\geq 0$.
    
    Let $\ell$ be an integer satisfying $\max{\cbra{0,2i-n}}\leq \ell\leq i$. We say $E$ is of type $\ell$ if $a_1(E)=\ell$. It is clear that for $h\in W_i$, the (naively-permissible) subsets $h(E)$ and $E$ have the same type. Note that $h\in W_i\sset S_{2n}^*$ decomposes into a (reduced) product of $\tau_{j'j}\coloneqq (j'j)(j^{'*}j^*)$ and $\tau_j\coloneqq (jj^*)$ for $1\leq j\leq 2n$. If $\cbra{j',j^{'*}}\cap E\neq \emptyset$, then $\cbra{j,j^*}\cap (\tau_{j'j}E)\neq \emptyset$. If $j\in E$, then $j^*\in (\tau_jE)$.  

    Let $E,E'$ be two naively-permissible subsets of type $\ell$. Using the action of $\tau_{j'j}$, we may assume that the pairs $\cbra{j,j^*}$ contributing one (resp. two) element(s) in $E$ and $E'$ are the same. Using the action of $\tau_j$, we may assume that $E=E'$. Hence, the set \begin{flalign}
    	\cbra{E\in\CS\ |\ a_1(E)=\ell}  \label{ESl}
    \end{flalign} is a $W_i$-orbit.
    
    From the above discussion, we conclude that there are precisely $\min{\cbra{i,n-i}}+1$ $W_i$-orbits in $\CS$.  
\end{proof}

For $\max{\cbra{0,2i-n}}\leq \ell\leq i$, the (naively-permissible) subset \begin{flalign*}
	E^\ell\coloneqq \sbra{i+1-\ell,n+i-\ell}\sset [1,2n]
\end{flalign*}
is a representative for the $W_i$-orbit in \eqref{ESl}. 
Then $E^\ell$ determines a point $(\CF^\ell_{i},\CF^\ell_{-i})$ in $\RM^\naive_i(k)$ via
\begin{equation}
	\begin{split}
		\CF_i^\ell &\coloneqq k\pair{f_{i+1-\ell},\ldots,f_{n+i-\ell},f_{2n+2}} \sset \Lambda_i\otimes k, \\ \CF_{-i}^\ell &\coloneqq k\pair{g_1,\ldots,g_{n-i+\ell},g_{i^*+\ell},\ldots,g_{2n}, g_{2n+1} }\sset \Lambda_{-i}\otimes k.
	\end{split}  \label{CF1234}
\end{equation}

\begin{corollary}\label{coro-bijection}
    The Schubert cells in $\RM^\naive_{\cbra{i},k}$ are in bijection with the orbits of naively-permissible subsets of $\sbra{1,2n}$ under the action of $W_i\simeq S_{2i}^*\times S_{2n-2i}^*$. Each Schubert cell in $\RM^\naive_{\cbra{i},k}$ admits a representative of the form $(\CF^\ell_{i},\CF^\ell_{-i})$ in \eqref{CF1234}.  In particular, there are precisely $\min{\cbra{i,n-i}}+1$ Schubert cells in $\RM^\naive_{\cbra{i},k}$. 
\end{corollary}
\begin{proof}
    Let $u=t^uu_0\in W_i$. Then $u\omega_i=\omega_i$ by Lemma \ref{lem-Wmfaces}. If $w$ is naively-permissible and $j\in E^w$, then \begin{flalign*}
    	uw\omega_i(u_0(j)) &=w\omega_i(j)+t^{u}(u_0(j))\\ &=\omega_i(j)+t^u(u_0(j)) \text{\quad (since $j\in E^w$)}\\ &= \omega_i(u_0(j)) \text{\quad (since $u\omega_i=\omega_i$).}
    \end{flalign*} 
    Hence, $u_0(j)\in E^{uw}$. By Lemma \ref{lem-Wmfaces}, $w\mapsto E^w$ defines a $W_i$-equivariant injective map \begin{flalign*}
    	\tcbra{\text{Naively-permissible elements in $\wt{W}^\circ$}}/W_i \hookrightarrow \CS=\cbra{ \begin{array}{l}
    		\text{Naively-permissible} \\ \text{subsets in $\sbra{1,2n}$}
    	\end{array} },
    \end{flalign*}
    and hence an injection on (left) $W_i$-orbit sets \begin{flalign*}
    	W_i\backslash\tcbra{\text{Naively-permissible elements in $\wt{W}^\circ$}}/W_i \hookrightarrow W_i\backslash \CS.
    \end{flalign*}
    By Lemma \ref{lem-cell}, the Schubert cells in $\RM^\naive_{\cbra{i},k}$ are in bijection with the left-hand orbit set. Thus, the number of Schubert cells in $\RM^\naive_{\cbra{i},k}$ is at most $|W_i\backslash \CS|=\min{\cbra{i,n-i}}+1$. On the other hand, by \eqref{CF1234}, there are at least $\min{\cbra{i,n-i}}+1$ Schubert cells in $\RM^\naive_{\cbra{i},k}$. This completes the proof. 
\end{proof}
\begin{remark}
	The proof of Corollary \ref{coro-bijection} also implies that for any naively-permissible subset $E$, there exists a naively-permissible $w\in \wt{W}^\circ$ such that $E=E^w$.
\end{remark}

\begin{prop}\label{prop-liftm4}
    Each of the $\min{\cbra{i,n-i}}+1$ points $(\CF^\ell_{i},\CF^\ell_{-i})\in \RM^\naive_{\cbra{i}}(k)$ has a lift to the generic fiber $\RM^\naive_{\cbra{i},F_0}$.
\end{prop}
\begin{proof}
    
    Note that for any $\CO_{F_0}$-algebra $R$, $\RM^\naive_{\cbra{i}}(R)$ is the set of pairs $(\CF_i,\CF_{-i})$ of $R$-modules such that 
    \begin{enumerate}[label=(\roman*)]
    	\item $\CF_i$ (resp. $\CF_{-i}$) is a locally direct summand of rank $n+1$ in $\Lambda_i\otimes k$ (resp. $\Lambda_{-i}\otimes k$);
    	\item $\CF_{-i}=\CF^\perp_i$;
    	\item Denote \begin{flalign*}
    		\lambda_1\coloneqq \begin{pmatrix}
    			\pi_0 I_i & & \\ &I_{2n-2i} & \\ & &\pi_0 I_i \\ & & &\begin{pmatrix} \pi_0 \\ &1 \end{pmatrix}
    		\end{pmatrix} \text{\ and\ } \lambda_2\coloneqq \begin{pmatrix}
    			I_i & & \\ &\pi_0 I_{2n-2i} & \\ & &I_i \\ & & &\begin{pmatrix} 1 \\ &\pi_0 \end{pmatrix}
    		\end{pmatrix},
    	\end{flalign*}
    	which are viewed as $R$-linear maps $\Lambda_{-i}\otimes_{\CO_{F_0}}R\ra \Lambda_i\otimes_{\CO_{F_0}} R$ and $\Lambda_i\otimes_{\CO_{F_0}} R\ra \Lambda_{-i}\otimes_{\CO_{F_0}} R$ respectively. We require $\lambda_1(\CF_{-i})\sset \CF_i$ and $\lambda_2(\CF_i)\sset \CF_{-i}$. 
    \end{enumerate} 
    Suppose that $\ell=i$. Define \begin{flalign*}
        \wt{\CF}^i_{i}\coloneqq \CO_{F_0}\pair{f_1,\ldots,f_{n},f_{2n+2} } \text{\ and\ } \wt{\CF}^i_{-i}\coloneqq \CO_{F_0}\pair{g_1,\ldots,g_{n},g_{2n+1} }
    \end{flalign*}
    Then $(\wt{\CF}^i_{i},\wt{\CF}^i_{-i})\in \RM^\naive_{\cbra{i}}(\CO_{F_0})$ lifts the point $(\CF^i_{i},\CF^i_{-i})$.

    Suppose that $\max{\cbra{0,2i-n}}\leq \ell<i$. Recall that $F=F_0(\pi)$ with $\pi^2=-\pi_0$. Define \begin{flalign*}
        \wt{\CF}^\ell_{i}\coloneqq \CO_F &\langle f_{i+1-\ell},\ldots,f_{n-i+\ell}, f_{n-i+\ell+1}+\pi f_{i-\ell}, f_{n-i+\ell+2}+\pi f_{i-\ell-1},\ldots, \\ &f_{n}+\pi f_1, f_{n+1}+\pi f_{2n}, f_{n+2}+\pi f_{2n-1},\ldots,f_{n+i-\ell}+\pi f_{2n-i+\ell+1}, f_{2n+2} \rangle.
    \end{flalign*}
    We have \begin{flalign*}
    	\wt{\CF}_{-i}^\ell = &\CO_F \langle g_1-{\pi}g_{n},g_2-{\pi}g_{n-1},\ldots,g_{i-\ell}-{\pi}g_{n-i+\ell+1},g_{i-\ell+1}, g_{i-\ell+2}, \\ &\ldots,g_{n-i+\ell}, g_{2n-i+\ell+1}-{\pi}g_{n+i-\ell}, g_{2n-i+\ell+2}-{\pi}g_{n+i-\ell-1}, \ldots, g_{2n}-{\pi}g_{n+1}, g_{2n+1}\rangle.
    \end{flalign*}
    Then it is easy to check that $(\wt{\CF}^\ell_{i},\wt{\CF}^\ell_{-i})\in \RM^\naive_{\cbra{i}}(\CO_F)$ and lifts $(\CF^\ell_{i},\CF^\ell_{-i})$.
\end{proof}

\begin{thm}\label{coro-topoflat}
    The naive local model $\RM^\naive_{\cbra{i}}$, and hence $\RM^\pm_{\cbra{i}}$, is topologically flat. 
\end{thm}
\begin{proof}
    Since the points $(\CF^\ell_{i},\CF^\ell_{-i})$ are representatives for (all) Schubert cells in $\RM^\naive_{\cbra{i},k}$ by Corollary \ref{coro-bijection}, it follows from Proposition \ref{prop-liftm4} that all points in $\RM^\naive_{\cbra{i},k}$ can be lifted to the generic fiber. Hence, $\RM^\naive_{\cbra{i}}$ is topologically flat. 
\end{proof} 

\begin{defn} \label{defn-Mell}
    Let $\iota\colon \Lambda_i\ra \pi_0\inverse\Lambda_{-i}$ denote the natural inclusion map (and its base change). 
    For an integer $\ell$, denote by $$S_i(\ell)\sset \RM^\naive_{\cbra{i},k}$$ the closed subscheme such that for any $k$-algebra $R$, $(\CF_i,\CF_{-i})\in \RM^\naive_{\cbra{i},k}(R)$ lies in $S_i(\ell)(R)$ if and only if $\wedge^{\ell+1}(\iota: \CF_i\ra \pi_0\inverse\CF_{-i})=0$. 
    
    Let $\RM^\naive_{\cbra{i}}(\ell)\sset \RM^\naive_{\cbra{i},k}$ denote the locus where $\iota(\CF_i)$ has rank $\ell$. 
\end{defn}


A similar proof as in \cite[Proposition 4.13]{yang25TopFlat} implies the following.

\begin{prop}\label{prop-stratificationMpm}
    There exists a stratification of the reduced special fiber \begin{flalign}
        (\RM^\naive_{\cbra{i},k})_\red=\coprod_{\ell=\max\cbra{0,2i-n}}^i \RM^\naive_{\cbra{i}}(\ell), \label{stratification}
    \end{flalign}
    where each stratum $\RM^\naive_{\cbra{i}}(\ell)$ is a single Schubert cell. Consequently, $\RM^\naive_{\cbra{i},k}$ is irreducible.
\end{prop}

\subsection{General parahoric case} \label{subsec-genepara}
Now we let $I\sset [0,n]$  be any non-empty subset. By construction, for each $i\in I$, there is a natural $L^+\CP_I$-equivariant  projection map \begin{flalign}
	  q_i: LG/L^+\CP_I\ra LG/L^+\CP_i.  \label{qiproj}
\end{flalign}
By the moduli description, we obtain the following.
\begin{lemma}\label{lem-intersection}
	The scheme $\RM^\naive_I$ equals the schematic intersection $\bigcap_{i\in I}q_i\inverse(\RM^\naive_{\cbra{i}})$. Here, $q\inverse_i(\RM^\naive_{\cbra{i}})$ denotes the pullback of $\RM^\naive_{\cbra{i}}$ along $q_i$. 
\end{lemma}

Denote by \begin{flalign}
	\Perm_I\sset \wt{W}^\circ \label{perm46}
\end{flalign}  the subset consisting of $w$ such that $C_w$ lies in $\RM^\naive_{I,k}$. Denote by \begin{flalign*}
        \Adm(\mu_\pm)\coloneqq \tcbra{w\in\wt{W}^\circ\ |\ w\leq \sigma t^{\mu_\pm}\sigma\inverse \text{\ for some $\sigma\in W^\circ$} }
        \end{flalign*}
    the $\mu_\pm$-admissible subset.

\begin{lemma} \label{lem-perm-admi}
	For any $i\in I$, we have \begin{flalign*}
		W_i\backslash \Perm_{i}/W_i = W_i\backslash \Adm(\mu_\pm)/W_i.
	\end{flalign*}
\end{lemma}
\begin{proof}   
	By Theorem \ref{coro-topoflat}, we obtain that $\RM^\naive_{\cbra{i},k}$ and $\RM^{\pm\loc}_{\cbra{i},k}$ (see Definition \ref{Mlocdefin}) have the same topological space. By Proposition \ref{propiso} and a well-known fact (see e.g. \cite[Theorem 1.1]{PZ}) about the schematic local model, we obtain the lemma. 
\end{proof}

For $i\in I$, there is an obvious projection map \begin{flalign*}
	\rho_i\colon W_I\backslash \wt{W}^\circ/W_I\ra W_i\backslash \wt{W}^\circ/W_i.
\end{flalign*}
Write $\Adm_i(\mu_\pm)\coloneqq W_i\backslash \Adm(\mu_\pm)/W_i$. We also use $\Perm_i$ to denote $W_i\backslash \Perm_i/W_i$ by abuse of notation. Since $q_i$ is $L^+\CP_I$-equivariant, we have \begin{flalign}
	W_I\backslash\Perm_I/W_I =\bigcap_{i\in I} \rho_i\inverse(\Perm_i)  \label{eq48}
\end{flalign}
by Lemma \ref{lem-intersection}.
Recall the following vertexwise criterion for the admissible sets.
\begin{thm}
	[{\cite[Theorem 1.5]{haines2017vertexwise}}] \label{thm-vertexcri}
	\begin{flalign*}
		W_I\backslash \Adm(\mu_\pm)/W_I = \bigcap_{i\in I}\rho_i\inverse (\Adm_i(\mu_\pm)).
	\end{flalign*}
\end{thm}

It follows that \begin{flalign*}
	 W_I\backslash\Perm_I/W_I &=\bigcap_{i\in I} \rho_i\inverse(\Perm_i) \text{\quad (by \eqref{eq48})}  \\ &=\bigcap_{i\in I}\rho_i\inverse(\Adm_i(\mu_\pm))  \text{\quad (by Lemma \ref{lem-perm-admi})} \\ &=W_I\backslash\Adm(\mu_\pm)/W_I \text{\quad (by Theorem \ref{thm-vertexcri})}.
\end{flalign*}  

\begin{thm}\label{coro-mainresults}
	The naive local model $\RM^\naive_I$, and hence $\RM^\pm_I$, is topologically flat for any non-empty $I\sset [0,n]$. 
\end{thm}
\begin{proof}
	By the previous discussion, the Schubert cells in $\RM^\naive_I$ are indexed by the admissible set $W_I\backslash\Adm(\mu_\pm)/W_I$. In particular, $\RM^\naive_I$ and $\RM^{\pm\loc}_I$ have the same topological space, and hence $\RM^\naive_I$ and $\RM^\pm_I$ are topologically flat. 
\end{proof}

\begin{corollary}\label{coro-ifreduced}
	If $\RM^\pm_{I,k}$ is reduced, then $\RM^\pm_I$ is flat over $\CO_F$.
\end{corollary}
\begin{proof}
	See \cite[Corollary 2.5]{Yang25}.
\end{proof}

\section{Affine charts}\label{AffineCharts}
\subsection{Affine chart for the naive local model}\label{Affine Charts naive}

In what follows, we fix $I=\{i\}$ for $0\leq i\leq n$. By Remark \ref{in-i}, we may assume that $0\leq 2i\leq n$. In view of periodicity, we may restrict the moduli functor ${\rm M}^{\rm naive}_I$ to the sub-lattice chain 
\begin{equation}
\begin{tikzcd}
\Lambda_{i, R}^\vee \arrow[r, "\lambda"] 
&
\Lambda_{i, R} \arrow[r, "\mu"] 
&
\pi_0^{-1}\Lambda_{i, R}^\vee
\\
\calF_i^\bot\arrow[u,hook]\arrow[r]
&
\calF_i\arrow[u,hook]\arrow[r]
&
\pi_0^{-1}\calF_i^\bot\arrow[u,hook]
\end{tikzcd},    
\end{equation}
where $\calF_i:=\calF_{\Lambda_i}$ and $\calF_i^\bot:=\calF_{\Lambda_i^\vee}$. The special fiber $\Mnaive_I\otimes k$ admits a natural embedding into a partial affine flag variety $\sFl_I$ associated to the orthogonal similitude group $\GO(V,\phi)$. According to \S \ref{subsec-pseud}, there exists a unique closed Schubert cell in $\Mnaive_I\otimes k$, which is denoted by $\Mnaive_{\cbra{i}}(0)\subset \Mnaive_I$ as in Proposition \ref{prop-stratificationMpm}. In our setting, the unique closed Schubert cell $\Mnaive_{\cbra{i}}(0)$ can be explicitly characterized by the standard lattices $\Lambda_{i}$ and its dual. We fix the point $*=(\calF_{i, 0}^\bot, \calF_{i, 0})\in\Mnaive_{\cbra{i}}(0)$, where
\begin{equation}\label{worst point}
\begin{array}{c}
\calF_{i, 0}^\bot={k}\bb 
	e_{1},\dots, e_i, \pi_0 e_{2n+1-i},\dots, \pi_0 e_{2n}, e_{i+1},\dots, e_{n-i}, f_1\pp,
  \\
\calF_{i, 0}={k}\bb
	e_{i+1},\dots, e_{n+i}, f_2\pp.
\end{array}
\end{equation}

By construction, the naive local model $\Mnaive_I$ can be viewed as a closed subscheme of a product of Grassmannians
\[
\Mnaive_{\{i\}
}\hookrightarrow \Gr(n+1,\Lambda_{i})\times \Gr(n+1,\Lambda_{i}^\vee).
\]
Let $\U^{\Gr}_i$ denote the standard open affine neighborhood in
$\Gr(n+1,\Lambda_{i})\times \Gr(n+1,\Lambda_{i}^\vee)$ containing the point
$(\calF_{i,0},\calF_{i,0}^\bot)$ corresponding to $*=(\calF_{i,0}^\bot,\calF_{i,0})$.
We then define the affine charts
\[
\calU_i^{\rm naive}:=\U^{\Gr}_i\cap \Mnaive_{\{i\}
},\qquad
\CU^\pm_i:=\U^{\Gr}_i\cap \RM^\pm_{\{i\}}.
\]

Recall that the parahoric group scheme $\sG^\circ_I$ (see  \eqref{Gpara}) acts on $\RM^\pm_{\cbra{i}}$.  Since $*$ lies in the unique closed Schubert cell of $\RM^\naive_{\cbra{i},k}$, we obtain the following. 
  \begin{lemma}\label{lem-Gtrans}
     The $\sG^\circ_{I}$-translates of $\CU^\pm_i$ cover $\RM^\pm_{\{i\}}$.   
  \end{lemma}
  In particular, to determine local properties of $\RM^\pm_{\cbra{i}}$, it suffices to determine the local coordinate ring of $\CU^\pm_i$. To simplify the computations, we rearrange the bases of $\Lambda_{i}$ (resp.\ $\Lambda_{i}^\vee$)
in the following manner:
\begin{equation}\label{reorderbasis}
\begin{array}{c}
\Lambda_{i}^\vee={\CO_{F_0}}\bb
	e_{1},\dots, e_i, \pi_0 e_{2n+1-i},\dots, \pi_0 e_{2n}, e_{i+1},\dots, e_{2n-i}, f_1, f_2\pp,  \\
\Lambda_{i}={\CO_{F_0}}\bb
	\pi_0^{-1}e_{1},\dots, \pi_0^{-1} e_i, e_{2n+1-i},\dots, e_{2n}, e_{i+1},\dots, e_{2n-i}, \pi_0^{-1} f_1, f_2\pp.
\end{array}
\end{equation}


Next, we derive the equations defining $\calU_i^{\rm naive}$ within the naive local model. For the sake of simplicity, we assume that $0\leq 2i\leq n$ as the case $2i>n$ is strictly analogous. A point $(\calF_i^\bot, \calF_i)$ in this affine chart $\U_{I}^{\rm naive}\subset \Mnaive_I$ can be presented by the following block matrices with respect to the basis order (\ref{reorderbasis}):
 \[  
\calF_i^\bot = \begin{blockarray}{cccc}
\matindex{2i} &\matindex{n-2i}&\matindex{1}&\\
\begin{block}{(ccc)c}
  I_{2i}  &0 & 0 & \matindex{2i} \\ 
 0 &I_{n-2i}&0& \matindex{n-2i} \\
 A_1 & A_2 & C_1 & \matindex{2i}\\
  A_3 & A_4 & C_2 & \matindex{n-2i}\\
   0 & 0 & 1 & \matindex{1}\\
 C_3 & C_4 & c & \matindex{1}\\
\end{block}
   \end{blockarray}, \quad 
\calF_i = \begin{blockarray}{cccc}
\matindex{n-2i} &\matindex{2i}&\matindex{1}&\\
\begin{block}{(ccc)c}
  B_2  &B_1 & D_1 & \matindex{2i} \\ 
 I_{n-2i} &0&0& \matindex{n-2i} \\
 0 & I_{2i} & 0 & \matindex{2i}\\
  B_4 & B_3 & D_2 & \matindex{n-2i}\\
 D_4 & D_3 & d & \matindex{1}\\
  0 & 0 & 1 & \matindex{1}\\
\end{block}
   \end{blockarray}.
\]
Note that the inclusion maps $\lambda: \Lambda_{i, R}^\vee\rightarrow \Lambda_{i, R},\,\, \mu: \Lambda_{i, R}\rightarrow \pi_0^{-1}\Lambda_{i, R}^\vee$ are given by:
\begin{equation}
\lambda=\left(
\begin{matrix}
\pi_0 I_{2i} &&\\ 
& I_{2n-2i}&\\
&& \left(\begin{array}{cc}
   \pi_0  &  \\
     & 1
\end{array}\right)
\end{matrix}\ \right),\quad	
\mu=\left(
\begin{matrix}
I_{2i} &&\\ 
& \pi_0 I_{2n-2i}&\\
&& \left(\begin{array}{cc}
   1  &  \\
     & \pi_0
\end{array}\right)
\end{matrix}\ \right).
\end{equation}
The symmetric pairing $\phi: \Lambda_{i}^\vee \times \Lambda_{i} \rightarrow R$ is given by
\begin{equation}
\phi=\left( 
\begin{matrix}
H_{2i}&&\\ 
& H_{n-2i} &\\
&&\left(\begin{array}{cc}
   1  &  \\
     & 1
\end{array}\right)
\end{matrix}\ \right),
\end{equation}
where $H_k$ is the unit anti-diagonal matrix of size $k$. We will omit the lower indices of $H_{k}$ if there is no confusion.

We first consider the isotropic condition $\phi(\calF_i^\bot, \calF_i)=0$, i.e., $\calF_{i}^\bot$ is the orthogonal complement of $\calF_{i}$. By using the above matrices, it translates to:
\[
\left(\begin{array}{ccc}
I_{2i}  &0 & 0  \\ 
0 &I_{n-2i}&0 \\
A_1 & A_2 & C_1 \\
A_3 & A_4 & C_2 \\
0 & 0 & 1 \\
C_3 & C_4 & c 
\end{array}\right)^t\cdot
\left( 
\begin{array}{ccc}
H_{2i}&&\\ 
& H_{n-2i} &\\
&&\left(\begin{array}{cc}
   1  &  \\
     & 1
\end{array}\right)
\end{array}\ \right)\cdot
\left(\begin{array}{ccc}
  B_2  &B_1 & D_1  \\ 
 I_{n-2i} &0&0\\
 0 & I_{2i} & 0 \\
  B_4 & B_3 & D_2 \\
 D_4 & D_3 & d \\
  0 & 0 & 1 
\end{array}\right)=0.
\]
This amounts to
\begin{equation}\label{eq 516}
\begin{array}{ccccc}
 B_1=-A_1^\ad,    & B_2=-A_3^\ad, &B_3=-A_2^\ad,&  B_4=-A_4^\ad,&\\
D_1=-HC_3^t,    & D_2=-HC_4^t, &D_3=-C_1^tH,&  D_4=-C_2^tH,& d=-c.
\end{array}
\end{equation}
Here $X^\ad:=HX^tH$ denotes the adjoint matrix. Consequently, $\calF_i$ is determined by $\calF_i^\bot$.

Second, we consider the inclusion map. The condition $\lambda(\calF_i^\bot)\subset \calF_i$ is equivalent to
\[
\left(\begin{array}{ccc}
\pi_0 I_{2i}  &0 & 0  \\ 
0 &I_{n-2i}&0 \\
A_1 & A_2 & C_1 \\
A_3 & A_4 & C_2 \\
0 & 0 & \pi_0 \\
C_3 & C_4 & c 
\end{array}\right)=
\left(\begin{array}{ccc}
  B_2  &B_1 & D_1  \\ 
 I_{n-2i} &0&0\\
 0 & I_{2i} & 0 \\
  B_4 & B_3 & D_2 \\
 D_4 & D_3 & d \\
  0 & 0 & 1 
\end{array}\right)\cdot
\left(\begin{array}{ccc}
 0 & I_{n-2i} & 0 \\
  A_1 & A_2 & C_1 \\
 C_3 & C_4 & c 
\end{array}\right).
\]
By equations (\ref{eq 516}), the above relation amounts to
\begin{equation}\label{tran eq1}
\begin{array}{cc}
A_1^\ad A_1+HC_3^tC_3=-\pi_0 I,    & A_2^\ad A_1+HC_4^tC_3=-A_3, \\
C_1^t HA_1+c C_3=0,    & A_4+A_4^\ad +A_2^\ad A_2 +HC_4^tC_4=0,\\
C_2^tH+C_1^tHA_2+cC_4=0, & C_1^tHC_1+c^2=-\pi_0.
\end{array}
\end{equation}
Similarly, the condition $\mu(\calF_i)\subset\pi_0^{-1}\calF_i^\bot$ is equivalent to 
\[
\left(\begin{array}{ccc}
  B_2  &B_1 & D_1  \\ 
 \pi_0 I_{n-2i} &0&0\\
 0 & \pi_0 I_{2i} & 0 \\
  \pi_0 B_4 & \pi_0 B_3 & \pi_0 D_2 \\
 D_4 & D_3 & d \\
  0 & 0 & \pi_0 
\end{array}\right)=
\left(\begin{array}{ccc}
I_{2i}  &0 & 0  \\ 
0 &I_{n-2i}&0 \\
A_1 & A_2 & C_1 \\
A_3 & A_4 & C_2 \\
0 & 0 & 1 \\
C_3 & C_4 & c 
\end{array}\right)\cdot
\left(\begin{array}{ccc}
 B_2 & B_1 & D_1 \\
  \pi_0 I_{n-2i}& 0 & 0 \\
 D_4 &D_3 & d 
\end{array}\right),
\]
which amounts to
\begin{equation}\label{trans eq2}
\begin{array}{cc}
A_1 A_3^\ad-\pi_0 A_2+C_1C_2^tH=0, & A_3A_3^\ad-\pi_0 A_4+C_2C_2^tH=\pi_0 A_4^\ad,\\
-C_3A_3^\ad+\pi_0 C_4-cC_2^tH=0, & A_1A_1^\ad +C_1C_1^tH=-\pi_0 I,\\
A_3A_1^\ad+ C_2C_1^tH=\pi_0 A_2^\ad, & C_3A_1^\ad+cC_1^tH=0,\\
A_3H C_3^t+c C_2=\pi_0 HC_4^t, & C_3HC_3^t+c^2=-\pi_0.
\end{array}
\end{equation}

\begin{prop}\label{simplifynaive}
For $0\leq 2i\leq n$, the affine chart $\U_{i}^{\rm naive}$ is isomorphic to 
\[
\mathbb{A}_{\CO_{F_0}}^{\frac{(n-2i)(n+2i+1)}{2}}\times  \Spec \CO_{F_0}[X]/\CI^\naive,
\]
where $\CI^\naive=\bb X^tH'X+\pi_0 H', XH' X^t+\pi_0 H'\pp$, and 
\[
H':=\left(\begin{array}{cc}
  H_{2i}   &  \\
     & 1
\end{array}\right).
\]
\end{prop}

\begin{proof}
We need to simplify equations (\ref{tran eq1}) and (\ref{trans eq2}). Let 
$X$ be a $(2i+1)\times (2i+1)$ matrix partitioned as 
\[
X:=\left(\begin{array}{cc}
  A_1   & C_1 \\
   C_3  & c
\end{array}\right).
\]
The relations $X^tH'X+\pi_0 H'=0$ and $XH' X^t+\pi_0 H'=0$ are equivalent to
\begin{equation}\label{eq I}
\begin{array}{cc}
  A_1^\ad A_1+HC_3^tC_3=-\pi_0 I,   & A_1A_1^\ad +C_1C_1^tH=-\pi_0 I, \\
   C_1^t HA_1+c C_3=0,  & C_3A_1^\ad+cC_1^tH=0,\\
     C_1^tHC_1+c^2=-\pi_0, & C_3HC_3^t+c^2=-\pi_0.
\end{array}   
\end{equation}
Note that equations (\ref{eq I}) are already contained within (\ref{tran eq1}), (\ref{trans eq2}). We claim that these constitute all the relations defining $\calU_I^{\rm naive}$. To see this, consider first the relation $A_4+A_4^\ad +A_2^\ad A_2 +HC_4^tC_4=0$ in (\ref{tran eq1}). Let the entries of the matrices be denoted by $A_4=(a_{i, j})_{1\leq i, j\leq n-2i}$, $A_2=(b_{i, j})_{1\leq i\leq 2i, 1\leq j\leq n-2i}$, $C_4=(c_i)_{1\leq i\leq n-2i}$. The above equation translates to
\[
a_{i, j}+a_{n-2i+1-j, n-2i+1-i}+\sum_{k=1}^{2i} b_{2i+1-k,n-3i+1}b_{k,j}+c_{n-3i+1}c_{j}=0,
\]
for $1\leq i, j\leq n-2i$. It follows that the entries $a_{i,j}$ with $i+j\geq n-2i+1$ are uniquely determined by $a_{i,j}$ with $i+j<n-2i+1$, $A_2$ and $C_4$. Secondly, regarding the relations 
\begin{equation}\label{eq 4110}
A_2^\ad A_1+HC_4^tC_3=-A_3,\quad C_2^tH+C_1^tHA_2+cC_4=0,   
\end{equation}
found in (\ref{tran eq1}), it is easy to see that the entries in $A_3, C_2$ are determined by $A_1$, $A_2$, $C_1$, $C_3$, $C_4$, $c$. The remaining equations are then automatically satisfied by virtue of (\ref{eq I}), (\ref{eq 4110}). Consequently, let $\CI^\naive$ be the ideal generated by the entries of the matrices $X^tH'X+\pi_0 H', XH' X^t+\pi_0 H'$. The free variables of the affine chart $\calU_I^{\rm naive}$ will be $a_{i,j}$ ($i,j<n-2i+1$), $A_2$ and $C_4$. The total number of these free variables is
\[
(1+\dots+ (n-2i))+2i(n-2i)+(n-2i)=\frac{(n-2i)(n+2i+1)}{2}.
\]

\end{proof}

To simplify the expression of the matrix $X$ and the subsequent computations, 
we adopt a new basis of $\Lambda_i^\vee$ in (\ref{reorderbasis}):
\begin{equation}\label{new basis}
\begin{array}{ll}
 \Lambda_i^\vee&={\CO_{F_0}}\bb g_1, \dots, g_{2n+2}\pp\\
 &={\CO_{F_0}}\bb
	e_{1},\dots, e_n, f_1, f_2, e_{n+1},\dots, e_{2n-i}, \pi_0 e_{2n+1-i}, \dots \pi_0 e_{2n}\pp.\\
\end{array}
\end{equation}
Under the new basis (\ref{new basis}), Proposition \ref{simplifynaive} can be rewritten as

\begin{prop}\label{prop-simplifynaive}
For $0\leq 2i\leq n$, the affine chart $\U_{i}^\naive$ is isomorphic to 
\[
\mathbb{A}_{\CO_{F_0}}^{\frac{(n-2i)(n+2i+1)}{2}}\times  \Spec \CO_{F_0}[X]/\CI^\naive,
\]
where $\CI^\naive=\bb X^tHX+\pi_0 H, XH X^t+\pi_0 H\pp$.
\end{prop}

\begin{proof}
 $\calF_i^\bot$ can be represented by the block matrix 
\begin{equation}\label{F_i matrix}
\calF_i^\bot = \begin{blockarray}{ccccc}
\matindex{i} &\matindex{n-2i}&\matindex{1} & \matindex{i}&\\
\begin{block}{(cccc)c}
  I_{i}  &0 & 0& 0 & \matindex{i} \\ 
 0 &I_{n-2i}&0& 0& \matindex{n-2i} \\
 A_1^1 & A_2^1 & C_1^1 & A_1^2& \matindex{i}\\
 0&0&1&0& \matindex{1}\\
  C_3^1 & C_4 & c & C_3^2& \matindex{1}\\
   A_1^3 & A_2^2 & C_1^2 & A_1^4& \matindex{i}\\
  A_3^1 & A_4 & C_2 & A_3^2& \matindex{n-2i}\\
 0 &0 & 0&  I_{i}  & \matindex{i} \\ 
\end{block}
   \end{blockarray},    
\end{equation}
where 
\[
A_1=\left(\begin{array}{cc}
   A_1^1  & A_1^2 \\
   A_1^3  & A_1^4
\end{array}\right), \quad
A_2=\left(\begin{array}{c}
   A_2^1 \\
   A_2^2  
\end{array}\right),\quad
A_3=\left(\begin{array}{cc}
   A_3^1  & A_3^2 
\end{array}\right),
\]
\[
C_1=\left(\begin{array}{c}
   C_1^1   \\
   C_1^2  
\end{array}\right),\quad 
C_3=\left(\begin{array}{cc}
   C_3^1  & C_3^2 
\end{array}\right).
\]
Thus, the matrix $X$ in Proposition \ref{simplifynaive} will be rewritten as 
\[
X:=\left(\begin{array}{ccc}
   A_1^1  &C_1^1& A_1^2 \\
   C_3^1 &c & C_3^2\\
   A_1^3  & C_1^2&A_1^4
\end{array}\right),
\]
under the new basis (\ref{new basis}), which satisfies $X^tH_{2i+1}X=-\pi_0 H_{2i+1}, XH_{2i+1}X^t=-\pi_0 H_{2i+1}$.
    
\end{proof}

\begin{prop} \label{moduli-sv}
    Suppose $0\leq 2i\leq n$. For $0\leq \ell\leq i $, the moduli functor $S_i(\ell)$ in Definition \ref{defn-Mell} is represented by a Schubert variety in $\RM^\naive_{\cbra{i},k}$. 
\end{prop}
\begin{proof}
    By Proposition \ref{prop-stratificationMpm}, the Schubert cells in $\RM^\naive_{\cbra{i},k}$ are described set-theoretically by the rank condition in Definition \ref{defn-Mell}. We are reduced to show that $S_i(\ell)$ is reduced. It is enough to prove that $\CU^\naive_i\cap S_i(\ell)$ is reduced. By the proof of Proposition \ref{simplifynaive}, the condition $\wedge^{\ell+1}(\iota:\CF_i\ra \pi_0\inverse\CF_{-i})=0$ in Definition \ref{defn-Mell} amounts to $\wedge^{\ell+1}X=0$ over $\CU^\naive_i\cap S_i(\ell)$. Hence, the affine scheme $\CU^\naive_i\cap S_i(\ell)$ is isomorphic to \begin{flalign*}
        \BA_k^{\frac{(n-2i)(n+2i+1)}{2}}\times \Spec \rR_\ell, \text{\ where\ }\rR_\ell\coloneqq  \frac{k[X]}{(X^tHX, XHX^t, \wedge^{\ell+1}X)}.
    \end{flalign*}
    By \cite[Theorem 1.11(1)]{Yang25}, the ring $\rR_\ell$ is reduced. Thus, $\CU^\naive_i\cap S_i(\ell)$ is reduced. 
\end{proof}

\subsection{Affine chart for the spin local model}\label{affine chart for spin}
We now turn our attention to the spin local model $\Mpm_{I}$. We continue with $I=\{i\}$ for $0\leq 2i\leq n$ as before. To evaluate the spin condition within the affine chart $\calU^\pm_I\subset \U^{\rm naive}_I\otimes_{\CO_{F_0}}\CO_{F}$, we first fix a split orthogonal basis in $V_F$, i.e., 
\begin{equation}
\begin{array}{ll}
 V_F  &=F\bb  e_1', \dots, e_{2n+2}'\pp\\
     &={F}\bb
	e_{1},\dots, e_n, f_1+\pi f_2, \frac{f_1-\pi f_2}{2\pi_0}, e_{n+1},\dots, e_{2n}\pp,
\end{array}
\end{equation}
where $\phi(e_i', e_{2n+3-j}')=\delta_{i,j}$.

In order to describe the basis of $\wedge^{n+1} V_F$ and $W_\pm$, we introduce the following notation. Recall that $[1,n]:=\{1,\dots, n\}$. Let $\calB=\{S\subset [1,2n+2] \mid \# S=n+1 \}$. For any integer $i\in S$, we define $i^*:= 2n+3-i$, and 
\[
S^*:=\{i^*\mid i\in S\},\quad S^\bot:=[1,2n+2] \setminus S^*.
\]
We also define $\sum S:= \sum_{i\in S}i$. For $S=\{i_1<\dots < i_{n+1}\}\subset [1,2n+2]$, let
\[
e_S':=e_{i_1}'\wedge \dots \wedge e_{i_{n+1}}'\in \wedge^{n+1} V_F,
\]
and let $\sigma_S$ be the permutation on $[1,2n+2]$ sending $[1,n+1]$ to $S$ in increasing order and $[n+2,2n+2]$ to $[1,2n+2]\setminus S$ in increasing order. The set $\{e_S'\mid S\in \calB\}$ forms a basis of $\wedge^{n+1} V_F$. The involution operator $a$ on $\wedge^{n+1} V_F$ is defined by
\[
a(e_S'):=\sgn(\sigma_S)e_{S^\perp}',
\]
where $\sgn(\sigma_S)$ is denoted by the sign of $\sigma_S$. By \cite[Lemma 2.8]{Sm3}, we have
\[
\sgn(\sigma_S)=(-1)^{\sum S+\lceil \frac {n+1}{2}\rceil}.
\]
It is straightforward to verify that $\sgn(\sigma_S)=\sgn(\sigma_{S^\bot})$. For the decomposition $\wedge^{n+1}V_F=W_+\oplus W_-$, we have
\begin{equation}
W_\pm=F\bb e_S'\pm \sgn(\sigma_S)e_{S^\bot}'\mid S\in \calB\pp.
\end{equation}
Furthermore, set $g_S=g_{i_1}\wedge \dots \wedge g_{i_{n+1}}$ where $\{g_j\}$ is the basis of $\Lambda_i^\vee$ defined in (\ref{new basis}). Note that
$\{g_S\mid S\in\calB\}$ forms a basis of $\wedge^{n+1}\Lambda_i^\vee$.

We now formulate the spin condition. Consider the intersections $(\wedge^{n+1}\Lambda_i^\vee)_\pm=\wedge^{n+1}\Lambda_i^\vee \cap W_\pm$. To describe the basis of $(\wedge^{n+1}\Lambda_i^\vee)_\pm$, we need to select exactly one representative from each pair of the form $\{e_S'\pm \sgn(\sigma_S)e_{S^\bot}', e_{S^\bot}'\pm \sgn(\sigma_{S^\bot})e_{S}'\}$. For any $S\subset [1,2n+2]$, define
\begin{equation}
d_S:=\#(S\cap [2n+3-i, 2n+2]).
\end{equation}
It follows that $d_{S^\bot}=\#(S^\bot\cap [2n+3-i, 2n+2])=i-\#(S\cap [1,  i])$. We write $S\succcurlyeq S^\bot$ if 
\begin{itemize}
    \item $S\neq S^\bot$, and $S$ is greater than $S^\bot$ in lexicographic order, or
    \item $S= S^\bot$, and $\sgn(\sigma_S)=1$ for $W_+$; $\sgn(\sigma_S)=-1$ for $W_-$.
\end{itemize}
The set $\{e_S'\pm \sgn(\sigma_S)e_{S^\bot}'\mid S\in \calB, S\succcurlyeq S^\bot\}$ forms a basis of $W_\pm$. To analyze the relationship between the basis elements $e_S'$ and $g_S$, we define the following index sets:
\begin{itemize}
    \item If $n+1\in S$, $n+2\notin S$, let $S_1:= (S\setminus \{n+1\})\cup \{n+2\}$ and $S_1^\bot:= (S^\bot\setminus \{n+1\})\cup \{n+2\}$.
    \item If $n+1\notin S$, $n+2\in S$, let $S_2:= (S\setminus \{n+2\})\cup \{n+1\}$ and $S_2^\bot:= (S^\bot\setminus \{n+2\})\cup \{n+1\}$
\end{itemize}
We then distinguish the following two cases:

{\it Case 1}: For $d_S> d_{S^\bot}$, we get 4 types of $S$:
\begin{itemize}
    \item Type I: Assume $n+1, n+2\in S$ (equivalently, $n+1, n+2\notin S^\bot$). We have
\begin{equation*}
  \alpha_1(S)=e_S'\pm \sgn(\sigma_S)e_{S^\bot}'=\frac{1}{\pi\cdot\pi_0^{d_S}} g_S\pm \sgn(\sigma_S)\frac{1}{\pi_0^{d_{S^\bot}}}g_{S^\bot}.
\end{equation*}
 \item Type II: Assume $n+1\in S, n+2\notin S$ (equivalently,  $n+1\in S^\bot, n+2\notin S^\bot$). We have
\begin{equation*}
  \alpha_2(S)=e_S'\pm \sgn(\sigma_S)e_{S^\bot}'=(\frac{1}{\pi_0^{d_S}} g_S+\frac{\pi}{\pi_0^{d_S}}g_{S_1})\pm \sgn(\sigma_S)(\frac{1}{\pi_0^{d_{S^\bot}}}g_{S^\bot}+\frac{\pi}{\pi_0^{d_{S^\bot}}}g_{S_1^\bot}).
\end{equation*}
 \item Type III: Assume $n+1\notin S, n+2\in S$ (equivalently,  $n+1\notin S^\bot, n+2\in S^\bot$). We have
 \begin{flalign*}
\alpha_3(S)&= e_S'\pm \sgn(\sigma_S)e_{S^\bot}'\\
&=\frac{1}{2}[(\frac{1}{\pi_0^{d_S+1}} g_{S_2}-\frac{\pi}{\pi_0^{d_S+1}}g_{S})\pm \sgn(\sigma_S)(\frac{1}{\pi_0^{d_{S^\bot}+1}}g_{S_2^\bot}-\frac{\pi}{\pi_0^{d_{S^\bot}+1}}g_{S^\bot})].
 \end{flalign*}
\quash{
\begin{equation*}
\alpha_3(S)= e_S'\pm \sgn(\sigma_S)e_{S^\bot}' 
 =\frac{1}{2}[(\frac{1}{\pi_0^{d_S+1}} g_{S_2}-\frac{\pi}{\pi_0^{d_S+1}}g_{S})\pm \sgn(\sigma_S)(\frac{1}{\pi_0^{d_{S^\bot}+1}}g_{S_2^\bot}-\frac{\pi}{\pi_0^{d_{S^\bot}+1}}g_{S^\bot})].
\end{equation*}
}
 \item Type IV: Assume $n+1, n+2\notin S$ (equivalently, $n+1, n+2\in S^\bot$). We have
\begin{equation}
\alpha_4(S)=  e_S'\pm \sgn(\sigma_S)e_{S^\bot}'=\frac{1}{\pi_0^{d_S}} g_S\pm \sgn(\sigma_S)\frac{1}{\pi\cdot\pi_0^{d_{S^\bot}}}g_{S^\bot}.
\end{equation}
\end{itemize}
   
Note that the above elements are not linearly independent in $W_\pm$. The basis elements $h_S$ for $(\wedge^{n+1}\Lambda_i^\vee)_\pm$
 are obtained by scaling these forms by a suitable power of $\pi$ for Type I (resp. Type IV). More precisely, we define $h_S=\pi\cdot\pi_0^{d_S}\alpha_1(S)$ (resp. $h_S=\pi_0^{d_S}\alpha_4(S)$). For S of Type II (resp. Type III), we define 
 \[
 h_S=\frac{\pi_0^{d_S}\alpha_2(S)+2\pi_0^{d_S+1}\alpha_3(S)}{2},\quad (\text{resp.}~ h_S=\frac{\pi_0^{d_S}\alpha_2(S)-2\pi_0^{d_S+1}\alpha_3(S)}{2\pi}).
\]
In summary, we define the elements $h_S$ for the case $d_S> d_{S^\bot}$ as follows:
\begin{equation}\label{h_S 1}
h_S=\begin{cases}
g_S\pm \sgn(\sigma_S) \pi\cdot \pi_0^{d_S-d_{S^\bot}}g_{S^\bot}, & \text{if}\ n+1, n+2\in S \ \text{(Type I)}\\
g_S\pm \sgn(\sigma_S) \pi\cdot \pi_0^{d_S-d_{S^\bot}}g_{S_1^\bot}, & \text{if}\ n+1\in S, n+2\notin S \ \text{(Type II)}\\
g_S\mp \sgn(\sigma_S) \frac{\pi_0^{d_S-d_{S^\bot}}}{\pi}g_{S_2^\bot}, & \text{if}\ n+1\notin S, n+2\in S \ \text{(Type III)}\\
g_S\pm \sgn(\sigma_S) \frac{\pi_0^{d_S-d_{S^\bot}}}{\pi}g_{S^\bot}, & \text{if}\ n+1, n+2\notin S \ \text{(Type IV)}
\end{cases}
\end{equation}
Recall that $\pi_0=-\pi^2$ and $d_S> d_{S^\bot}$, so that $\frac{\pi_0^{d_S-d_{S^\bot}}}{\pi}\in \CO_{F}$.

{\it Case 2}: For $d_S= d_{S^\bot}$, we get 2 types of $S$:
\begin{equation}\label{h_S 2}
h_S=\begin{cases}
g_S\pm \sgn(\sigma_S) \pi g_{S^\bot}, & \text{if}\ n+1, n+2\in S \ \text{(Type V)}\\
g_S\pm \sgn(\sigma_S) \pi g_{S_1^\bot}, & \text{if}\ n+1\in S, n+2\notin S \ \text{(Type VI)}
\end{cases}.
\end{equation}
We want to mention that when $d_S= d_{S^\bot}$ and $n+1\notin S, n+2\in S$, the element $h_S=g_{S_2^\bot}\mp \sgn(\sigma_S) \pi g_{S}$, is of the same form as those categorized as Type VI. Similarly, for $n+1, n+2\notin S$, we obtain $h_S=g_{S^\bot}\pm \sgn(\sigma_S) \pi g_{S}$, which is the same as Type V. By combining (\ref{h_S 1}) and (\ref{h_S 2}), we obtain

\begin{prop}
Set $\calB_0:=\{S\in \calB\mid S ~\text{of {\rm Type I--VI}}~\}$. We have
\[
\{h_S\mid S\in \calB_0\}
\]
forms an $\CO_{F}$-basis of $(\wedge^{n+1}\Lambda_i^\vee)_\pm$.
\end{prop}

The spin condition requires that $\wedge^{n+1}\calF_i^\bot=\sum_{S\in \calB} a_S g_S\in (\wedge^{n+1}\Lambda_i^\vee)_\pm$, where $a_S$ represents the determinants of the $(n+1)\times (n+1)$ submatrix formed by selecting the rows indexed by $S$ in increasing order in (\ref{F_i matrix}). We can express this sum as
\[
\sum_{S\in \calB} a_S g_S=\sum_{S\in \calB_0} c_S h_S,
\]
for $c_S\in \CO_{F}$. Consequently, for index sets $S$ belonging to Type I--VI, the above relation is equivalent to
\begin{equation}\label{spin-relation}
\begin{array}{ll}
  a_{S^\bot}=\pm \sgn(\sigma_S) \pi\cdot \pi_0^{d_S-d_{S^\bot}} a_S,   &\text{for}~ S ~\text{of type I},~  \\
   a_{S_1^\bot}=\pm \sgn(\sigma_S) \pi\cdot \pi_0^{d_S-d_{S^\bot}} a_S,   &\text{for}~ S ~\text{of type II},~  \\
   a_{S_2^\bot}=\mp \sgn(\sigma_S) \frac{\pi_0^{d_S-d_{S^\bot}}}{\pi} a_S,   &\text{for}~ S ~\text{of type III},~  \\
   a_{S^\bot}=\pm \sgn(\sigma_S) \frac{\pi_0^{d_S-d_{S^\bot}}}{\pi} a_S,   &\text{for}~ S ~\text{of type IV},~  \\
   a_{S^\bot}=\pm \sgn(\sigma_S) \pi a_S,   &\text{for}~ S ~\text{of type V},~  \\
   a_{S_1^\bot}=\pm \sgn(\sigma_S) \pi a_S,   &\text{for}~ S ~\text{of type VI}.  
\end{array}    
\end{equation}

To simplify the aforementioned relation, we establish a connection between the $(n+1)\times (n+1)$ minors of $\wedge^{n+1}\calF_i^\bot$ and the  $i\times i$ minors of the matrix $X$. For any subset $U\subset [1,  2i+1]$ with cardinality $\# U=i$, we adopt notation analogous to that introduced for $S\in \calB$. Specifically, we define $U^*=\{2i+2-i\mid i\in U\}$ and let $U^\bot=[1,  2i+1] \setminus U^*$ denote its complement. Furthermore, let $\sigma_U$ be the permutation on $[1,  2i+1]$ sending $[1,  i]$ to $U$ in increasing order and $[i+1,  2i+1]$ to $[1,  2i+1] \setminus U$ in increasing order.

\begin{lemma}
We have 
\[
\sgn(\sigma_U)=(-1)^{\sum U+\lceil \frac {i}{2}\rceil}.
\]
\end{lemma}

\begin{proof}
The proof is similar to \cite[Lemma 2.8]{Sm3}. 
Let $P$ denote the permutation matrix attached to $\sigma_U$. Using the Laplace expansion, we obtain that
\[
\det(P)=(-1)^{\sum U+(1+\dots +i)}=(-1)^{\sum U+\lceil \frac {i}{2}\rceil}.
\]
\end{proof}
We now simplify the relations in (\ref{spin-relation}) for index sets $S$ of type III-VI.
When $S$ is of type III, we have $d_{S}> d_{S^\bot}$, $n+1\notin S, n+2\in S$, and $ n+1\in S_2^\bot, n+2\notin S_2^\bot$. Let $[i+1,  n-i]\subset S$ and let $d_S-d_{S^\bot}=1$, i.e., 
\[
\#(S\cap ([1, i]\cup [2n+3-i, 2n+2]))=i+1.
\]
The spin condition corresponds to choosing $i$ rows in the matrix $X$, i.e., $\#(S\cap( [n-i+1,  n]\cup\{ n+2\}\cup [n+3, n+2+i]))=i$. Set
\[
\begin{array}{ll}
\calS_1:=S\cap [1,i],     & \calS_2:=S\cap [n-i+1, n], \\
\calS_3:=S\cap [n+3,  n+2+i],     & \calS_4:=S\cap [2n-i+3,  2n+2],
\end{array}
\]
Denote by $r_i:=\#\calS_i$ for $1\leq i\leq 4$. By assumption, these cardinalities satisfy $r_1+r_4=i+1$, $r_2+r_3=i-1$. Utilizing the Laplace expansion, the minor $a_S$ is given by
\[
a_S=(-1)^N\cdot \det(I_{n-i+1})\cdot [U:U'],
\]
where $[U:U']$ denotes the determinant of the $i\times i$ submatrix of 
$X$ with row indices $U=((\calS_2-(n-i))\cup\{i+1\}\cup (\calS_3-(n+1-i))$ and column indices $U'=([1, i]\setminus \calS_1)\cup\{i+1\}\cup (([2n-i+3,  2n+2]\setminus \calS_4)-(n+1))$. The exponent $N$ is determined by 
\begin{flalign*}
N=& \frac{r_1(r_1+1)}{2}+\sum \calS_1+\frac{(2r_1+n+1-2i)(n-2i)}{2}+\frac{(n+1)(n-2i)}{2}\\
&+\frac{r_4(2n+3-r_4)}{2}+\sum \calS_4-r_4(n+1).
\end{flalign*}
Similarly, for the dual set $S_2^\bot$, we have $\#\calS_1^\bot=i-r_4=r_1-1$, $\#\calS_2^\bot=i-r_3=r_2+1$, $\#\calS_3^\bot=i-r_2=r_3+1$, $\#\calS_4^\bot=i-r_1=r_4-1$. Consequently,
\[
a_{S_2^\bot}=(-1)^{N'} \cdot \det(I_{n-i})\cdot [U^\bot:{U'}^{\bot}],
\]
where $[U^\bot:{U'}^\bot]$ is the determinant of a size $(i+1)$ submatrix of $X$. Here
\begin{flalign*}
N'=& \frac{r_1(r_1-1)}{2}+\sum \calS_1^\bot+\frac{(2r_1+n-1-2i)(n-2i)}{2}+\frac{(n+1)(n-2i)}{2}\\
&+\frac{(r_4-1)(2n+4-r_4)}{2}+\sum \calS_4^\bot-(r_4-1)(n+1)\\
&+(r_1+r_2+n-2i+1)+(n-i+1).
\end{flalign*}
Note that $\sum \calS_1^\bot+\sum \calS_4^\bot=\sum \calS_1+\sum \calS_4-(2n+3).$
Therefore, $(-1)^{N'-N}=(-1)^{r_1+r_2+n+1}$. We obtain
\begin{equation}\label{eq III}
[U^\bot:{U'}^\bot]=\pm (-1)^{r_1+r_2+n+1}\sgn(\sigma_S)\pi [U:U'],
\end{equation}
when $S$ is of type III.

For $S$ of type IV, we have $n+1\notin S, n+2\notin S$. Under the assumptions $[i+1,  n-i]\subset S$ and $d_S-d_{S^\bot}=1$. We choose $i$ rows in the matrix $X$ such that $r_1+r_4=i+1$, $r_2+r_3=i$. By using the same method, we get
\begin{equation}\label{eq IV}
[U^\bot:{U'}^\bot]=\pm (-1)^{r_1+r_2+n+1}\sgn(\sigma_S)\pi [U:U'].
\end{equation}

If $S$ is of type V, then $n+1\in S, n+2\in S$. Assuming $[i+1,  n-i]\subset S$ and $d_S=d_{S^\bot}$, we select $i$ rows in the matrix $X$ such that $r_1+r_4=i$, $r_2+r_3=i-1$. Note that $\sum \calS_1^\bot+\sum \calS_4^\bot=\sum \calS_1+\sum \calS_4$ by $r_1+r_4=i$, which yields:
\begin{equation}\label{eq V}
[U^\bot:{U'}^\bot]=\pm (-1)^{i+r_1+r_2}\sgn(\sigma_S)\pi [U:U'].
\end{equation}

Finally, for $S$ of type VI, we have $n+1\in S, n+2\notin S$. Let $[i+1,  n-i]\subset S$ and let $d_S=d_{S^\bot}$. We choose $i$ rows in the matrix $X$ such that $r_1+r_4=i$, $r_2+r_3=i$.  We get
\begin{equation}\label{eq VI}
[U^\bot:{U'}^\bot]=\pm (-1)^{i+r_1+r_2}\sgn(\sigma_S)\pi [U:U'].
\end{equation}

The relations in (\ref{eq III})--(\ref{eq VI}) can be uniformly represented as follows:

\begin{thm}\label{SpinConditions}
The spin condition in $\Mpm_{I}$ gives
\begin{equation}\label{eq-spincondition}
[U^\bot:{U'}^\bot]=\pm (-1)^{i}\sgn(\sigma_U)\sgn(\sigma_{U'})\pi [U:U']. 
\end{equation}
Here $[U:U']$ denotes the minor of $X$ whose rows (resp. columns) are given by $U$ (resp. $U'$) in increasing order.
\end{thm}

\begin{proof}
The correspondence between index sets $S$ of type III-VI and the structure of the minors $[U:U']$ is summarized in the following table:
\[
\begin{array}{|c|c|}
\hline
  S ~\text{of type III}   &  i+1\in U, i+1\in U' \\ \hline
  S ~\text{of type IV}   &  i+1\notin U, i+1\in U' \\ \hline
  S ~\text{of type V}   &  i+1\in U, i+1\notin U' \\ \hline
  S ~\text{of type VI}   &  i+1\notin U, i+1\notin U' \\ \hline
\end{array}
\]
As $S$ ranges over types III-VI, the associated minors $[U:U']$ exhaust all possible $i\times i$ minors of the matrix $X$. Let the parity exponent $l(S)$ be defined by: 
\[
l(S)=\frac{(n+1)(n-2i)}{2}+\frac{(n+2)(n+1)}{2}+(n-i)r_2+(n-i+1)r_3-(2n+3)i+(2n-2i+1)(i-r_4).
\]
Direct calculation of the permutation signs yields the following ratios:
\[
\frac{\sgn(\sigma_S)}{\sgn(\sigma_{U})\sgn(\sigma_{U'})}=
\begin{cases}
(-1)^{l(S)+(n+2)-2(i+1)}, & \text{if $S$ of type III}, \\
(-1)^{l(S)-(i+1)}, & \text{if $S$ of type IV}, \\
(-1)^{l(S)+(n+1)+(n+2)-(i+1)}, & \text{if $S$ of type V}, \\
(-1)^{l(S)+(n+1)}, & \text{if $S$ of type VI}. \\
\end{cases}
\]
Note that the exponent simplifies to $(-1)^{l(S)}=(-1)^{(n+ni+i+1)+(n-i)(r_2+r_3)+r_3+r_4}$. In particular, for $S$ of type III, we have $\frac{\sgn(\sigma_S)}{\sgn(\sigma_{U})\sgn(\sigma_{U'})}=(-1)^{n+i+1+r_3+r_4}$. Substituting this into (\ref{eq III}), the relation reduces to:
\begin{equation}\label{eq 4218}
[U^\bot:{U'}^\bot]=\pm (-1)^{i}\sgn(\sigma_U)\sgn(\sigma_{U'})\pi [U:U'].    
\end{equation}
A straightforward verification shows that the relations for the remaining cases (\ref{eq IV})--(\ref{eq VI}) all specialize to the same unified equation (\ref{eq 4218}). This finishes the proof of the theorem.
\end{proof}

We introduce the refinement of  $\U^{\rm naive}_I$ by adding the relation (\ref{eq-spincondition}) in Theorem \ref{SpinConditions}. This defines a subscheme $\U_I$ that fits into a diagram of closed immersions
\begin{equation}\label{eq-closedimm}
\U_I^\pm\subset \calU_I\subset\calU_I^{\rm naive}\otimes_{\CO_{F_0}}\CO_{F}.
\end{equation}
\begin{prop}\label{AffineSplitChart}
We have 
\[
\calU_I=\mathbb{A}_{\CO_{F}}^{\frac{(n-2i)(n+2i+1)}{2}}\times  \Spec \CO_{F}[X]/\calI.
\]
Here $\CI=\pair{X^tHX+\pi_0 H, XH X^t+\pi_0 H }+\CI^\pm = \CI^\naive + \CI^\pm$, and $\CI^\pm$ denotes the ideal generated by $[U^\bot:{U'}^\bot]=\pm (-1)^{i}\sgn(\sigma_U)\sgn(\sigma_{U'})\pi [U:U']$, where $U,U'$ run through all subsets in $[1,2i+1]$ of cardinality $i$.
\end{prop}

\begin{remark}
{\rm
(1). In the next section, we will prove that $\CU_I$ equals $\CU^\pm_I$ and is $\CO_F$-flat. Consequently, it suffices to consider the relations indexed by $S$ in (\ref{eq-spincondition}). The relations corresponding to $S$ of type I or II, as well as the remaining relations of type III–IV, are automatically satisfied.

(2). We consider the spin condition (LM4$\pm$) only for $\calF_i^\bot$-part. In fact, one can prove that the spin condition for $\wedge^{n+1}\calF_i\in (\wedge^{n+1}\Lambda_i)_\pm$ is equivalent to $\wedge^{n+1}\calF_i^\bot\in (\wedge^{n+1}\Lambda_i^\vee)_\pm$, see \cite[Proposition 2.4.3]{Luo}.
}    
\end{remark}

\section{Flatness of the spin local model}\label{FlatSec}

\subsection{Maximal Parahoric Case}
In this subsection, we prove  that the spin local model $\RM^\pm_{\cbra{i}}$ is flat over $\CO_F$. 

\begin{prop}\label{prop-CRflat}
	The ring
\[
\CR_i\coloneqq \frac{\CO_F[X]}{(XHX^t,X^tHX,\wedge^{i+1}X)}
\] 
	is flat over $\CO_F$ and reduced; here the matrix $X$ is of size $(2i+1)\times (2i+1)$.
\end{prop} 
\begin{proof}
	This immediately follows from \cite[Theorem 1.11(1)]{Yang25}.
\end{proof}

\begin{lemma}\label{lem-irre}
	Set $\CR_{i,k}\coloneqq \CR_i\otimes_{\CO_F}k$. The scheme $\Spec \CR_{i,k}$ is irreducible.
\end{lemma}
\begin{proof}
	By (\ref{eq-closedimm}), we have \begin{flalign*}
		\CU^\pm_{i,k}\sset \rbra{\Spec\CR_{i,k}\times_k(\text{affine space})} \sset  \CU^\naive_{i,k}.
	\end{flalign*}
	By Theorem \ref{coro-topoflat}, the two inclusions above induce homeomorphisms on the underlying topological spaces. By Proposition \ref{prop-stratificationMpm}, $\rbra{\Spec\CR_{i,k}\times_k(\text{affine space})}$ is irreducible; hence so is $\Spec\CR_{i,k}$.
\end{proof}

\begin{prop}\label{prop-genred}
	 The $k$-scheme $\Spec \CR_{i,k}$ is generically smooth, i.e., there exists an open dense subscheme of $\Spec \CR_{i,k}$ smooth over $k$.
\end{prop}
\begin{proof}
	Write $X$ as a block matrix \begin{flalign*}
		\begin{array}{c@{}c@{}c@{}c}
          X= &\left(\begin{array}{ccc}
               A &E_1 &B  \\
               E_2 &e &E_3 \\
               C &E_4 &D
          \end{array}\right) &\begin{array}{c}
               \raisebox{-0.5ex}{\mbox{\tiny $i$}}\vspace{5pt}  \\
               \raisebox{0.5ex}{\mbox{\tiny $1$}} \\
               \raisebox{0.5ex}{\mbox{\tiny $i$}}
          \end{array}  \\
          &\begin{array}{ccc}
              \raisebox{0.5ex}{\mbox{\ \tiny $i$}} &\raisebox{0.5ex}{\mbox{\quad \tiny $1$}} &\raisebox{0.5ex}{\ \mbox{\tiny $i$}}  
          \end{array} 
     \end{array}.
	\end{flalign*}
	We claim that the open subscheme $\Spec \CR_{i,k}[\det(A)\inverse]$ is smooth over $k$.
	Set \begin{flalign*}
		P_1\coloneqq \begin{pmatrix}
			I_i \\ -E_2A\inverse &1 \\ -CA\inverse & &I_i
		\end{pmatrix} \text{\ and\ } P_2\coloneqq \begin{pmatrix}
			I_i &-A\inverse E_1 &-A\inverse B\\  &1 \\ & &I_i
		\end{pmatrix}.
	\end{flalign*}
	Then we have \[ P_1XP_2=\begin{pmatrix}
		A \\ &e-E_2A\inverse E_1 &E_3-E_2A\inverse B\\ &E_4-CA\inverse E_1 &D-CA\inverse B
	\end{pmatrix}.\]
	Under the transformation \begin{flalign*}
		B'\coloneqq A\inverse B, C'\coloneqq CA\inverse, D'\coloneqq D-CA\inverse B, E_1'\coloneqq A\inverse E_1,\\ E_2'\coloneqq E_2A\inverse, E_3'\coloneqq E_3-E_2A\inverse B, E_4'\coloneqq E_4-CA\inverse E_1, e'\coloneqq e-E_2A\inverse E_1,
	\end{flalign*}
	the ring $\CR_{i,k}[\det(A)\inverse]$ is isomorphic to \begin{flalign*}
	   \frac{k[A,B',C',D',E_1',E_2',E_3',E_4',e'][\det(A)\inverse]}{\CI'},
	\end{flalign*}
	where $\CI'$ is generated by (entries in) \begin{flalign*}
		 &B'H_i+H_iB^{\prime t}+E_1'E_1^{\prime t},\ H_iE_3^{\prime t}+E_1'e',\ e^{\prime 2},\ H_iD^{\prime t}+E_1'E_4^{\prime t},\ H_iE_4'+E_2^{\prime t}e'\\ &C'H_i+H_iC^{\prime t}+E_2'E_2^{\prime t},\ \wedge^{i+1}\begin{pmatrix}
		 	A\\ &e' &E_3'\\ &E_4' &D'
		 \end{pmatrix}.
	\end{flalign*}
	Since $A$ is invertible, the relation $$\wedge^{i+1}\begin{pmatrix}
		 	A\\ &e' &E_3'\\ &E_4' &D'
		 \end{pmatrix}=0$$ amounts to \begin{flalign*}
		 	\begin{pmatrix}
		 		e' &E_3'\\ E_4' &D'
		 	\end{pmatrix}=0.
		 \end{flalign*}
    Hence, $\CR_{i,k}[\det(A)\inverse]$ is isomorphic to \begin{flalign*}
    	\frac{k[A,B',C',E_1',E_2'][\det(A)\inverse]}{(B'H_i+H_iB^{\prime t}+E_1'E_1^{\prime t}, C'H_i+H_iC^{\prime t}+E_2'E_2^{\prime t})},
    \end{flalign*}
    whose spectrum is isomorphic to an open subscheme of $\BA^{i(2i+1)}_k$. In particular, the scheme $\Spec \CR_{i,k}[\det(A)\inverse]$ is smooth over $k$. By Lemma \ref{lem-irre}, $\Spec \CR_{i,k}[\det(A)\inverse]$ is an open dense subscheme of $\Spec \CR_{i,k}$. We finish the proof. 
\end{proof}

Let us recall the following version of Hironaka's lemma.
\begin{prop}[{\cite[Corollary 3]{kollar1995flatness}}]\label{prop-Hiron}
	Let $Z\ra \Spec\CO_F$ be a flat morphism of schemes. Assume that $Z$ has no embedded primes, the special fiber $Z_k$ is generically reduced, and the reduction $\red~Z_k$ is normal. Then $Z_k$ is reduced. 
\end{prop}

\begin{thm}\label{thm-maxflat}
	The spin local model $\RM^\pm_{\{i\}}$ is flat over $\CO_F$.
\end{thm}
\begin{proof}
	By Corollary \ref{coro-ifreduced}, it suffices to prove that $\RM^\pm_{i,k}$ is reduced. In turn, it is enough to prove that $\CU^\pm_{i,k}$ is reduced. Set $Z\coloneqq \Spec\CR_i$. By Proposition \ref{prop-CRflat}, the ring $\CR_i$ is flat over $\CO_F$ and has no embedded primes (see \cite[031Q, 031R]{stacks-project}). By Proposition \ref{prop-genred}, $\CR_{i,k}$ is generically reduced. As in the proof of Lemma \ref{lem-irre}, we obtain that $(\red~Z_k)\times_k(\text{affine space})$ is an open scheme of $\RM^{\pm\loc}_{i,k}$. Since $\RM^{\pm\loc}_{i,k}$ is irreducible, it is normal (and Cohen-Macaulay) by Proposition \ref{propiso} and \cite[Theorem 1.1]{PZ}. Therefore, $\red~Z_k$ is also normal. Applying Proposition \ref{prop-Hiron} to $Z=\Spec\CR_i$, we obtain that $Z_k$ is reduced. Since $\CU^\pm_{i,k}$ and $\CU^\naive_{i,k}$ have the same topological space, and \begin{flalign*}
		\CU^\pm_{i,k}\sset \rbra{Z_k\times_k(\text{affine space})} \sset  \CU^\naive_{i,k},
		\end{flalign*}
		we conclude that $\CU^\pm_{i,k}$ is reduced. 
\end{proof}

By the proof of the above theorem, we obtain the following.
\begin{corollary} \label{coro-Uspin}
    For $I=\cbra{i}$, the inclusion $\CU^\pm_I\sset \CU_I$ in \eqref{eq-closedimm} is an equality. In particular, $\CU^\pm_{i,k}$ is isomorphic to (the reduced scheme) \[
\mathbb{A}_k^{\frac{(n-2i)(n+2i+1)}{2}}\times \Spec \CR_{i,k}.
\] 
\end{corollary}

\begin{remark}\label{ExoticGoodReduction}
By Theorem \ref{thm-maxflat}, the spin local model $\Mpm_{\{i\}}$ is flat over $\CO_{F}$. In particular, when $I= \{0\}$, we have 
\[
\calU^{\pm}_{0}\simeq \mathbb{A}_{\CO_{F}}^{\frac{n(n+1)}{2}}\times  \Spec \CO_{F}[x]/\CI,
\]
where $\CI=\bb x^2-\pi^2, x\mp\pi\pp$ by Proposition \ref{AffineSplitChart}. Therefore, the affine chart $\calU^{\pm}_{0}\simeq \mathbb{A}_{\CO_{F}}^{\frac{n(n+1)}{2}}$ and hence $\Mpm_{\{0\}}$ is smooth over $\CO_{F}$. This is a case of ``exotic" good reduction (see also \cite[Theorem 1.2]{HPR}).    
\end{remark}

\subsection{General Parahoric Case} Let $I \subset [0,n]$ be non-empty. Then for each $i \in I$, the special fiber $\RM_{\{i\},k}^{\pm}$  of $\RM_{\{i\}}^{\pm}$ is reduced by the proof of Theorem \ref{thm-maxflat}. Recall that by construction, for each $i \in I$, we have a natural
$L^{+}\mathcal{G}^{\circ}_I$-equivariant commutative diagram
\[
\begin{tikzcd}
\RM_{I,k}^{\pm} \arrow[r, hook] \arrow[d] &
LG^{\circ}/L^{+}\mathcal{G}^{\circ}_I \arrow[d, "\rho_i"] \\
\RM_{\{i\},k}^{\pm} \arrow[r, hook] &
LG^{\circ}/L^{+}\mathcal{G}^{\circ}_{\{i\}}
\end{tikzcd}
\]
and we have
\begin{equation}
 \RM_{I,k}^{\pm}=\bigcap_{i\in I}\rho_i^{-1}\bigl(\RM_{\{i\},k}^{\pm}\bigr)
\end{equation}
as a schematic intersection. (Note that we denote by $ \RM_{I,k}^{\pm}$ the special fiber of $ \RM_{I}^{\pm}$.) The pullback $\rho_i^{-1}\bigl(\RM_{\{i\},k}^{\pm}\bigr)$ of
$\RM_{\{i\},k}^{\pm}$ along $\rho_i$ is reduced since it is a $L^{+}\mathcal{G}^{\circ}_i/L^{+}\mathcal{G}^{\circ}_I$-bundle. Also, $\rho_i^{-1}\bigl(\RM_{\{i\},k}^{\pm}\bigr)$
is a union of Schubert varieties in $LG^{\circ}/L^{+}\mathcal{G}^{\circ}_I$ since $\rho_i$ is $L^{+}\mathcal{G}^{\circ}_I$-equivariant.

\begin{thm} \label{thmflatspin}
$\Mpm_{I}$ is flat over $ \CO_{F} $.
\end{thm}
\begin{proof}
To prove flatness for $\Mpm_{I}$ is enough to show that it is topologically flat and its special fiber is reduced. From \S \ref{sec-topoflat}, we obtain topological flatness for $\Mpm_{I}$. Since $|\pi_1(G^{\circ\der})|=|\pi_1(\SO(V,\phi))|=2$ is prime to $p$ we conclude that any Schubert variety is normal \cite[Theorem 4.25]{FHLR} and compatibly Frobenius-split \cite[Theorem 4.1]{FHLR}. Therefore, arbitrary intersections of unions of Schubert varieties are Frobenius-split and so reduced by \cite[Corollary 2.7]{Go}. 
\end{proof}

\section{A resolution for a spin local model}\label{ResolSpin} 

Throughout Section~\ref{ResolSpin}, we assume that $I=\{1\}$. Our goal is to construct a resolution of the spin local model $\Mpm_{I}$ by blowing up $\Mpm_I$ along the unique closed Schubert cell in the special fiber
(see Definition~\ref{defn-Mell}). 

Let us introduce some notation that will be used in the proof of the theorem below: Recall from Proposition \ref{simplifynaive} the affine open subscheme $\U_{1}^{\rm naive}\otimes \CO_{F} \subset {\rm M}^{\rm naive}_I\otimes \CO_{F}$, containing the point $*=(\calF_{1, 0}^\bot, \calF_{1, 0})\in\Mnaive_1(0)$ defined in (\ref{worst point}), which is isomorphic to 
\[
\mathbb{A}_{\CO_{F}}^{\frac{(n-2)(n+3)}{2}}\times  U^{\rm naive}_1 \]
where $U^{\rm naive}_1 = \Spec R  $ and $ R:= \CO_{F}[X]/\CI^\naive$ with $ X := (x_{ij})_{1\leq i,j \leq 3}$ and 
$\CI^\naive=\bb X^tH_3X-\pi^2 H_3, XH_3 X^t-\pi^2 H_3\pp$. Recall that $H_3$ is the unit antidiagonal matrix of size $3 \times 3$. By adding the spin conditions (see Proposition~\ref{AffineSplitChart}), we obtain the closed subschemes $\calU^\pm_1\subset \U_{1}^{\rm naive}\otimes \CO_{F} $ and $ U^{\pm}_{1} \subset U^{\rm naive}_1$ where $ U^{\pm}_{1}  =\Spec \CO_{F}[X]/(\CI^\naive + \CI^{\pm}) $ and  
\[
\CI^{\pm} := (m^{1, 1}_{2,2} \pm\pi x_{11},\, m^{1, 1}_{2,3} \pm\pi x_{12},\,m^{1, 2}_{2,3} \pm \pi x_{13},\,m^{1, 1}_{3,2} \pm \pi x_{21},\, m^{1, 1}_{3,3} \pm \pi x_{22},
\]
\[
m^{1, 2}_{3,3} \pm\pi x_{23},\,m^{2, 1}_{3,2} \pm \pi x_{31},\,m^{2, 1}_{3,3} \pm \pi x_{32},\,m^{2, 2}_{3,3} \pm \pi x_{33}).
\]
Here, with $m^{i, j}_{t,s}$ we denote the minor 
 \[
m^{i, j}_{t,s}= \begin{pmatrix} x_{i, j} & x_{i, s}\\ 
x_{t, j}& x_{t, s}\\
\end{pmatrix}=
x_{i,j}x_{t,s} -x_{i,s}x_{t,j} .
\]



\begin{defn}\label{def.semistable}
{\rm Let $\CO$ be a discrete valuation ring and let $\CX$ be a locally noetherian $\CO$-scheme. We say that $\CX$ has \emph{semi-stable reduction} over $\CO$ if its special fiber $\CX_k$ is reduced and is a \emph{strict normal crossings divisor} on $\CX$ (in the sense of \cite[2.4]{Jong}).}
\end{defn}
\begin{remark}\label{rmk:sr}
If moreover $\CX_\eta$ is smooth, then $\CX$ is regular; see \cite[Rem.~1.1.1]{Ha}.
\end{remark}
 
\begin{thm}\label{ResolutionTheorem}
Assume $ I = \{1\}$. The blow-up $\Mbl_I$ of $\Mpm_{I}$ along the unique closed Schubert cell in the special fiber has semi-stable reduction. In particular, it is regular and its special fiber is a reduced divisor with two smooth irreducible components intersecting transversely. 
\end{thm}

\begin{proof}
From the above discussion and by Lemma \ref{lem-Gtrans}, it suffices to compute the blow-up of $\U_{1}^{\rm naive}\otimes \CO_{F}$ along
$\Mnaive_1(0)\cap (\U_{1}^{\rm naive}\otimes \CO_{F})$; equivalently, it suffices to blow up the affine chart $U^{\rm naive}_1$ along the ideal $J:=(X,\pi)$. We have the blow-up morphism $\rho: V^{\rm naive}_1\rightarrow U^{\rm naive}_1 $, where $V^{\rm naive}_1:={\rm Proj (\tilde{R})}.$ Here $\tilde{R}$ is the graded $R$-algebra $\oplus_{d\geq 0} J^d$ with $J^0=R$. Denote by $V^{\pm}_{1}$ the strict transform of $U^{\pm}_{1} $ under $ \rho$. We are then reduced to showing that $V_1^\pm$ is semi-stable. 
Moreover, by \S \ref{sec-naivespinmod} the generic fiber of $V^{\pm}_{1}$ is smooth. Hence, by Remark \ref{rmk:sr}, it is enough to show that the special fiber of $V^{\pm}_{1}$ is reduced and has strict normal crossings. 

There are $10$ affine charts in the blow-up $V^{\rm naive}_1$ of $\Spec R$ along the ideal $J$, where $R=\CO_F[X]/\CI^\naive$. Let $ y_{11},\dots y_{33},\alpha$ be the represented elements of $x_{11},\dots x_{33},\pi $ in the graded algebra $\tilde{R}$, considered as
homogeneous elements of degree 1. Denote by $ Y := (y_{ij})_{1\leq i,j \leq 3}$. By symmetry, it is enough to consider the following 4 affine charts: 

\textit{Case 1}: Assume $y_{11}=1$ (similarly for $y_{13}=1$, $y_{31}=1$ and $y_{33}=1$). Then, we have $X= x_{11} Y $ and $\pi = x_{11}\alpha $. The affine chart $D_{+}(y_{11}) \subset V^{\rm naive}_1$ is isomorphic to 

\[\frac{\CO_{F}[x_{11},Y,\alpha]}{(Y^tH_3Y-\alpha^2H_3, \, YH_3 Y^t-\alpha^2 H_3,\, y_{11} -1, \, \pi-x_{11}\alpha)}.
\]
By direct calculations the off anti-diagonal entries give 
\[
\begin{pmatrix}
y_{13}\\[2pt] y_{31}\\[2pt] y_{23}\\[2pt] y_{32}
\end{pmatrix}
=
\frac12\begin{pmatrix}
-\,y_{12}^2\\[2pt]
-\,y_{21}^2\\[2pt]
y_{12}\,(y_{12}y_{21}-2y_{22})\\[2pt]
y_{21}\,(y_{12}y_{21}-2y_{22})
\end{pmatrix}, \quad \frac{\Delta}{4}
\begin{pmatrix}
-\,y_{21}\\[2pt]
\ \,y_{21}^{2}\\[2pt]
-\,y_{12}\\[2pt]
\ \,y_{12}^{2}
\end{pmatrix}= \begin{pmatrix}
0\\[2pt]
 0\\[2pt]
0\\[2pt]
0
\end{pmatrix}
\]
where $\Delta:=(y_{12}y_{21}-2y_{22})^{2}-4y_{33}$. From the antidiagonal entries we obtain 
\[
\alpha^{2}=(y_{12}y_{21}-y_{22})^{2}, \quad \alpha^{2}= (y_{12}y_{21}-y_{22})^{2}-\frac{1}{4}\Delta.
\]
Combining these two equations we get $\Delta =0$ and thus the equations that we have from all the above are 
\begin{equation}\label{eq.prop.61}
\begin{pmatrix}
y_{13}\\[2pt] y_{31}\\[2pt] y_{23}\\[2pt] y_{32} \\[2pt] y_{33} 
\end{pmatrix}
=
\frac12\begin{pmatrix}
-\,y_{12}^2\\[2pt]
-\,y_{21}^2\\[2pt]
y_{12}\,(y_{12}y_{21}-2y_{22})\\[2pt]
y_{21}\,(y_{12}y_{21}-2y_{22})\\[2pt]
(y_{12}y_{21}-2y_{22})^{2}/2
\end{pmatrix}, \quad \alpha^{2}=(y_{12}y_{21}-y_{22})^{2}, \quad \pi=x_{11}\alpha.
\end{equation}
Now, the intersection $D_{+}(y_{11}) \cap V^{\pm}_{1}$ is obtained by combining the relations from (\ref{eq.prop.61}) and substituting  $X= x_{11} Y $ and $\pi = x_{11}\alpha $ in the relations of $\CI^{\pm}$. In particular, observe that $(m^{1, 1}_{2,2} \pm\pi x_{11})\in \CI^{\pm}$ and 
\[
m^{1, 1}_{2,2} \pm\pi x_{11}=x_{11}x_{22} - x_{12} x_{21} \pm\pi x_{11}= x^2_{11}(y_{22} -y_{12}y_{21} \pm \alpha) 
\]
and so in $D_{+}(y_{11}) \cap V^{\pm}_{1}$  we have $y_{22} -y_{12}y_{21}= \pm \alpha $. All the other relations from $\CI^{\pm}$ are satisfied from the above and thus $D_{+}(y_{11}) \cap V^{\pm}_{1}$ is isomorphic to 
\[
\Spec  \frac{\CO_{F}[x_{11},y_{12},y_{21},y_{22}]}{(x_{11}(y_{22} -y_{12}y_{21}) \pm \pi)}.
\]
We can easily see that the special fiber of this affine scheme is reduced and has strict normal crossings.

\textit{Case 2}: Assume $y_{12}=1$ (similarly for $y_{21}=1$, $y_{32}=1$ and $y_{23}=1$). Then $X=x_{12}Y$ and $\pi=x_{12}\alpha$. By direct calculation, the off anti-diagonal entries are equivalent to
\begin{equation}\label{rel602}
\begin{aligned}
2y_{11}y_{13}+1&=0,\\
y_{21}&=2y_{11}(y_{11}y_{23}+y_{22}),\\
y_{31}&=-2y_{11}(y_{11}y_{23}+y_{22})^{2},\\
y_{32}&=2y_{11}y_{23}(y_{11}y_{23}+y_{22}),\\
y_{33}&=y_{11}y_{23}^{2},
\end{aligned}
\end{equation}
and the remaining off anti-diagonal relations follow from these. From the anti-diagonal entries we obtain
\[
(\alpha\,y_{13})^{2}=(y_{13}y_{22}-y_{23})^{2}.
\]
Now intersect with $V^{\pm}_{1}$. Since $(m^{2,3}_{1,1}\pm \pi x_{12})\in \CI^{\pm}$ and
$m^{2,3}_{1,1}=x_{11}x_{23}-x_{13}x_{21}$, we get
\[
0=x_{11}x_{23}-x_{13}x_{21}\pm \pi x_{12}
= x_{12}^{2}(y_{11}y_{23}-y_{13}y_{21}\pm \alpha),
\]
hence $y_{11}y_{23}-y_{13}y_{21}=\mp \alpha$. Multiplying by $y_{13}$ and using $2y_{11}y_{13}+1=0$ we have
\[
y_{13}(y_{11}y_{23}-y_{13}y_{21})
= y_{13}(y_{22}+2y_{11}y_{23})
= y_{13}y_{22}-(2y_{11}y_{13})y_{23}
= y_{13}y_{22}-y_{23},
\]
so the spin condition forces $ \alpha\,y_{13}=\pm (y_{13}y_{22}-y_{23})$. Using $\pi=x_{12}\alpha$ we obtain
\[
x_{12}(y_{13}y_{22}-y_{23})\pm \pi y_{13}=0.
\]
Using $\alpha\,y_{13}=\pm (y_{13}y_{22}-y_{23})$ and (\ref{rel602}) the remaining generators of $\CI^{\pm}$ are automatically satisfied. Therefore,
\[
D_{+}(y_{12})\cap V^{\pm}_{1}\ \simeq\
\Spec \CO_{F}\bigl[x_{12},\,y_{13}^{\pm 1},\,y_{22},\,y_{23}\bigr]\Big/\bigl(x_{12}(y_{13}y_{22}-y_{23})\pm \pi y_{13}\bigr)
\]
and this has a reduced special fiber with strict normal crossings.

\textit{Case 3}: Assume $y_{22}=1$. Then $X=x_{22}Y$ and $\pi=x_{22}\alpha$.
The affine chart $D_{+}(y_{22})\subset V^{\rm naive}_1$ is isomorphic to
\[
\Spec \frac{\CO_{F}[x_{22},Y,\alpha]}{(Y^tH_3Y-\alpha^2H_3,\; YH_3Y^t-\alpha^2H_3,\; y_{22}-1,\; \pi-x_{22}\alpha)}.
\]
By direct calculation, the off anti-diagonal entries give the eight relations
\begin{equation}\label{eq:y22-off}
\begin{array}{ll}
(1)\ 2y_{11}y_{13}+y_{12}^{2}=0, &
(5)\ 2y_{11}y_{31}+y_{21}^{2}=0,\\[2pt]
(2)\ y_{11}y_{23}+y_{12}+y_{13}y_{21}=0, &
(6)\ y_{11}y_{32}+y_{12}y_{31}+y_{21}=0,\\[2pt]
(3)\ y_{21}y_{33}+y_{23}y_{31}+y_{32}=0, &
(7)\ y_{12}y_{33}+y_{13}y_{32}+y_{23}=0,\\[2pt]
(4)\ 2y_{31}y_{33}+y_{32}^{2}=0, &
(8)\ 2y_{13}y_{33}+y_{23}^{2}=0.
\end{array}
\end{equation}

On the anti-diagonal one obtains the four relations
\begin{equation}\label{eq:y22-anti}
\begin{array}{ll}
(9)\ \alpha^{2}=1+2y_{21}y_{23}, &
(11)\ \alpha^{2}=y_{11}y_{33}+y_{12}y_{32}+y_{13}y_{31},\\[2pt]
(10)\ \alpha^{2}=1+2y_{12}y_{32}, &
(12)\ \alpha^{2}=y_{11}y_{33}+y_{21}y_{23}+y_{13}y_{31}.
\end{array}
\end{equation}
Combining (9) with (11) gives
\[
y_{13}y_{31}+y_{12}y_{32}+y_{11}y_{33}=2y_{23}y_{21}+1,
\]
hence this chart is covered by $D(y_{11})\cup D(y_{12})\cup D(y_{21})\cup D(y_{13})$. (In contrast to Cases~1--2, on the $y_{22}$-chart we do not obtain a single global
elimination of the remaining variables from \eqref{eq:y22-off}--\eqref{eq:y22-anti}. Thus, we work on the cover $D(y_{11})\cup D(y_{12})\cup D(y_{21})\cup D(y_{13})$ to obtain the desired simplifications.)

We treat $D(y_{11})$ (the cases $D(y_{12})$ and $D(y_{21})$ are analogous).
On $D(y_{11})$, using \eqref{eq:y22-off}--\eqref{eq:y22-anti} one checks that the chart
simplifies to
\[
D_{+}(y_{22})\cap D(y_{11})\ \simeq\
\Spec \frac{\CO_{F}[y_{11},y_{12},y_{21},x_{22},\alpha][y_{11}^{-1}]}
{\bigl((y_{21}y_{12}y_{11}^{-1}-1)^2-\alpha^2,\ \pi-x_{22}\alpha\bigr)}.
\]
Now intersect with $V^{\pm}_{1}$. As in Case~1, use the generator
$(m^{1,1}_{2,2}\pm \pi x_{11})\in \CI^{\pm}$, where $m^{1,1}_{2,2}=x_{11}x_{22}-x_{12}x_{21}$.
Substituting $x_{ij}=x_{22}y_{ij}$ and $\pi=x_{22}\alpha$ we obtain
\[
0=m^{1,1}_{2,2}\pm \pi x_{11}
= x_{22}^2(y_{11}-y_{12}y_{21}\pm \alpha\,y_{11}).
\]
Thus on $D(y_{11})$ we have $y_{12}y_{21}-y_{11}=\pm \alpha\,y_{11},$
equivalently 
\[y_{12}y_{21}y_{11}^{-1}-1=\pm \alpha.\] Using $\pi=x_{22}\alpha$ we obtain $\pi \pm x_{22}(y_{21}y_{12}y_{11}^{-1}-1)=0$ and hence
\[
D_{+}(y_{22})\cap D(y_{11})\cap V^{\pm}_{1}\ \simeq\
\Spec \frac{\CO_{F}[y_{11},y_{12},y_{21},x_{22}][y_{11}^{-1}]}
{\bigl(\pi \pm x_{22}(y_{21}y_{12}y_{11}^{-1}-1)\bigr)}
\]
which has a reduced special fiber with strict normal crossings.
On $D(y_{13})$ one obtains similarly
\[
D_{+}(y_{22})\cap D(y_{13})\ \simeq\
\Spec \frac{\CO_{F}[y_{12},y_{13},y_{23},x_{22},\alpha][y_{13}^{-1}]}
{\bigl((y_{12}y_{23}y_{13}^{-1}-1)^2-\alpha^2,\ \pi-x_{22}\alpha\bigr)},
\]
and intersecting with $V^{\pm}_{1}$ gives
$y_{12}y_{23}y_{13}^{-1}-1=\pm \alpha$. Indeed, on $D(y_{13})$ we use the generator $(m^{1,2}_{2,3}\pm \pi x_{13})\in \CI^{\pm}$,
where $m^{1,2}_{2,3}=x_{12}x_{23}-x_{13}x_{22}$. Substituting $x_{ij}=x_{22}y_{ij}$
and $\pi=x_{22}\alpha$ gives
\[
0=m^{1,2}_{2,3}\pm \pi x_{13}
= x_{22}^2\bigl(y_{12}y_{23}-y_{13}\pm \alpha\,y_{13}\bigr),
\]
so on $D(y_{13})$ we obtain $y_{12}y_{23}y_{13}^{-1}-1=\pm \alpha$. Therefore,
\begin{equation}\label{chart13}
D_{+}(y_{22})\cap D(y_{13})\cap V^{\pm}_{1}\ \simeq\
\Spec \frac{\CO_{F}[y_{12},y_{13},y_{23},x_{22}][y_{13}^{-1}]}
{\bigl(\pi \pm x_{22}(y_{12}y_{23}y_{13}^{-1}-1)\bigr)}
\end{equation}
and its special fiber is a reduced divisor with strict normal crossings.

\textit{Case 4}: Assume $\alpha=1$. We have $ X = \pi Y$ and thus the affine chart $D_{+}(\alpha) \subset V^{\rm naive}_1$ is isomorphic to 
\[
\Spec \frac{\CO_{F} [Y]}{(Y^tH_3Y-H_3, \,YH_3 Y^t-H_3)}
\]
which in turn is isomorphic to the coordinate ring of the orthogonal group $O(3)$ over $\CO_{F}$ and thus smooth (as $\Char(k)\neq 2$). The scheme $O(3)$ has two connected components and the spin condition cuts out one of them; hence $V^{\pm}_{1}\cap D_+(\alpha)$ is smooth.

\end{proof}

\section{Applications}\label{ShimuraVarSec}

\subsection{Orthogonal integral models}   \label{ss51}
Let $(\bB,*)$ be a definite quaternion algebra over $\BQ$ equipped with a positive involution $*$. Let $(\bV,\pair{\ ,\ })$ be a free $\bB$-module of rank $n+1$, endowed with a $\BQ$-valued perfect alternating bilinear form $\pair{\ ,\ }$ which is skew-hermitian, i.e., $\pair{bv,w}=\pair{v,b^*w}$ for all $v,w\in \bV$, $b\in \bB$. Let $\bG=\bG(\bB,*,\bV,\pair{\ ,\ })$ denote the reductive group over $\BQ$ whose $R$-points for any $\BQ$-algebra $R$ are given by \begin{flalign*}
    \bG(R)\coloneqq \cbra{g\in \GL_{\bB\otimes_\BQ R}(\bV\otimes_\BQ R)\ |\ \pair{gv,gw}=c(g)\pair{v,w} \text{\ for some\ } c(g)\in R\cross }.
\end{flalign*} 
We have $\bG_{\ol{\BQ}}\simeq \GO_{2n+2,\ol{\BQ}}$.
Let $\bh\colon\Res_{\BC/\BR}\BG_m\ra \bG_\BR$ be a homomorphism which defines on $\bV_\BR$ a Hodge structure of type $\cbra{(-1,0),(0,-1)}$. The pairing $\pair{-,\bh(i)-}$ is a symmetric positive definite bilinear form on $\bV_\BR$. 

Then we obtain a PEL-datum $(\bB,*,\bV,\pair{\ ,\ },\bh)$ of type $D_{n+1}$. Set \begin{flalign*}
	\bE\coloneqq \BQ\rbra{\Tr(b\ |\ V^{-1,0})\ |\ b\in \bB }.
\end{flalign*}
As in the proof of \cite[Lemma 6.1]{Yang25}, we obtain $\bE=\BQ$.

Let $p$ be an odd prime. Assume that \begin{enumerate}
	\item [(A1)] $B\coloneqq \bB\otimes_\BQ\BQ_p\simeq M_2(\BQ_p)$.
\end{enumerate}  
Let $V\coloneqq \bV_{\BQ_p}=V_1\oplus V_2$ denote the Morita decomposition as in \cite[Definition 6.1]{Yang25}. Then $V_1$ is a $2n+2$-dimensional vector space over $\BQ_p$. Let $C\coloneqq \begin{psmallmatrix}
    0 &1\\ -1 &0
\end{psmallmatrix}$ be the Weil element in $B$ (identified with $M_2(\BQ_p)$) so that the positive involution $*$ on $B$ is given by $A^*=CA^tC\inverse$. Define \begin{flalign*}
    (x ,y)\coloneqq \pair{x, Cy}
\end{flalign*}
for $x,y\in V$. By \cite[Lemma 6.5]{Yang25}, we have \begin{flalign*}
	\bG_{\BQ_p}\simeq \GO(V_1,(\ ,\ )).
\end{flalign*}
Assume that \begin{enumerate}
	\item [(A2)] \label{condA2} there exists a $\BQ_p$-basis of $V_1$ such that the induced matrix corresponding to the pairing $(\ ,\ )$ is  \begin{flalign*}
	  \begin{pmatrix}
	  	H_{2n} \\ &\begin{pmatrix} p \\ &1 \end{pmatrix}
	  \end{pmatrix}.
\end{flalign*}
\end{enumerate} 
Denote $\CO_B=M_2(\BZ_p)$. Let $\sL$ be the self-dual (with respect to $\pair{\ ,\ }$) $\CO_B$-lattice chain in $V$ whose first component in the Morita decomposition is given by a self-dual (with respect to $(\ ,\ )$) $\BZ_p$-lattice chain $\CL$.
Let $\sG$ denote the (affine smooth) group scheme of similitude automorphisms of $\CL$. Set \begin{equation*}
    \bK_p\coloneqq \sG(\BZ_p)\sset \bG(\BZ_p).
\end{equation*}

Applying the results of the previous sections to the case $F_0\coloneqq \BQ_p$ and $F\coloneqq \BQ_p(\sqrt{-p})$, we obtain the spin local model $$\RM^\pm_\CL\ra \Spec\CO_F,$$ which is flat with reduced special fiber. 
\begin{defn}\label{Mpmm}
    Denote by $\RM_\sL^\pm$ the projective scheme over $\CO_F$ representing the functor that sends an $\CO_F$-scheme $S$ to the set of $\CO_S$-modules $(\CF_\Lambda)_{\Lambda\in\sL}$, where $\CF_{\Lambda}$ is a locally free $\CO_S$-submodule of  $\Lambda_S\coloneqq \Lambda\otimes_{\BZ_p}\CO_S$  of rank $2n+2$ such that \begin{enumerate}[label=(\roman*)]
        \item for any $\Lambda\in\sL$, the $\CO_S$-submodule $\CF_\Lambda\sset \Lambda_S$ is Zariski locally direct summand;
        \item for any $\Lambda\in\sL$, $\CF_\Lambda$ is invariant under the $\CO_B$-action on $\Lambda_S$;
        \item for any inclusion $\Lambda\sset \Lambda'$ in $\sL$,  the natural maps $\Lambda_S\ra \Lambda'_S$ induced by $\Lambda\hookrightarrow\Lambda'$ sends $\CF_\Lambda$ to $\CF_{\Lambda'}$; and the isomorphism $\Lambda_S\simto(p\Lambda)_S$ induced by $\Lambda\xrightarrow{p}p\Lambda$ identifies $\CF_\Lambda$ with $\CF_{p\Lambda}$;
        \item the perfect skew-Hermitian pairing $\varphi\colon \Lambda_S\times \Lambda^\#_S\ra \CO_S$ 
        induced by $\pair{\ ,\ }$ satisfies $\varphi(\CF_\Lambda,\CF_{\Lambda^\#})=0$, where $\Lambda^\#$ denotes the dual $\CO_B$-lattice of $\Lambda$ with respect to $\pair{\ ,\ }$;
        \item For any $\Lambda\in\CL$, denote by $(\CF_{\Lambda})_1$ (resp. $(\Lambda_S)_1$) the first component in the Morita decomposition of $\CF_\Lambda$ (resp. $\Lambda_S$). Then $(\CF_{\Lambda})_1$ is locally a direct summand of $(\Lambda_S)_1$ of $\CO_S$-rank $n+1$. We require that $\CF_\Lambda$ satisfies (LM4$\pm$) in Definition \ref{defn-spin}. 
    \end{enumerate}
\end{defn}

As discussed in \cite[\S 6]{Yang25}, we have an isomorphism \[\RM^\pm_{\sL} \simto \RM^\pm_\CL\]
by Morita equivalence.
We recall the following definition of Rapoport and Zink; cf. \cite[\S 6.8-6.9]{RZbook} and \cite[Definition 6.7]{Yang25}.
\begin{defn} \label{defn-naiveS}
     Let $\bK^p$ be a sufficiently small open compact subgroup in $\bG(\BA_f^p)$. Set $\bK\coloneqq \bK_p\bK^p$.
	Denote by $\CA^\pm_{\bK}$ the quasi-projective $\CO_F$-scheme representing the functor $$\CA^\pm_{\bK}\colon (\Sch/\CO_F)^\op\lra \Sets$$ which sends any $\CO_F$-scheme $S$ to the set of isomorphism classes of the following data.
	\begin{enumerate}[label=(\alph*)]
		\item An $\sL$-set of abelian $\CO_B$-schemes $(A_\Lambda,\iota_\Lambda)_{\Lambda\in\sL}$ in $AV(S)$ such that $A_\Lambda$ is a $(2n+2)$-dimensional abelian scheme over $S$.
		\item A $\BQ$-homogeneous polarization $\ol{\lambda}$ of the $\sL$-set $(A_\Lambda,\iota_\Lambda)_{\Lambda\in\sL}$.
		\item A $\bK^p$-level structure $$\ol{\eta}\colon H_1(A,\BA^p_f)\simeq \bV\otimes_\BQ\BA_f^p\mod \bK^p$$ that respects the bilinear forms on both sides up to a constant in $(\BA_f^p)\cross$.
		\item We further impose the following condition on $(A_\Lambda,\iota_\Lambda)_{\Lambda\in\sL}$: Choose an \etale cover $T\ra S$ such that there exists an isomorphism \begin{flalign*}
        \phi\colon M((A_{\Lambda})_T)\simeq \Lambda_T
    \end{flalign*} of $\CO_B\otimes_{\BZ_p}\CO_T$-modules, where $M((A_{\Lambda})_T)$ denotes the $\CO_T$-linear dual of the first de Rham cohomology of $(A_\Lambda)_T$.  We require that the image of the covariant Hodge filtration $\phi(\Fil^1((A_\Lambda)_T))$ in $\Lambda_T$ satisfies the spin condition in the sense of condition (v) in Definition \ref{Mpmm}.
	\end{enumerate}
\end{defn}

We will sometimes say that $\CA_\bK^\pm$ is a moduli space of type $D$. As in \cite[\S 6]{Yang25}, we have the local model diagram 
\[
\CA_\bK^\pm\xleftarrow{\alpha'_\bK}\wt{\CA}_\bK^\pm\xrightarrow {\beta'_\bK}\RM^\pm_\CL,
\]
where $\alpha'_\bK$ is a $\sG$-torsor, and $\beta'_\bK$ is $\sG$-equivariant and smooth of relative dimension $\dim G$. Equivalently, there is a relatively representable smooth morphism of algebraic
stacks
 \[
\phi: \CA_\bK^\pm \to [ \RM^\pm_\CL/ \sG\otimes_{\mathbb{Z}_p}\CO_F]
 \]
which is smooth of relative dimension $\text{dim}\,G$. 

\begin{corollary}\label{coro-AKflat}
    The moduli space $\CA_{\bK}^\pm$ of type $D$ is a normal, Cohen-Macaulay, flat $\CO_F$-scheme with reduced special fiber.
\end{corollary}

\begin{remark}
    As in \cite[Remark 6.8]{Yang25}, the determinant condition (Kottwitz condition) in \cite[Definition 6.9]{RZbook} is automatically satisfied; and due to the failure of the Hasse principle for $\bG$, the Shimura variety in our setting is just an open and closed subscheme of the generic fiber $\CA^\pm_{\bK,F}$.
\end{remark}

Now assume $\bK_1\coloneqq \bK'_1\bK^p$  where $\bK'_1 = \sG_1(\BZ_p)$ and $\sG_1(\BZ_p)$ is the stabilizer of the lattice chain $\CL = \Lambda_{\{1\}}$ (see \S \ref{ParahoricSubsection} for notation). Let us form the cartesian product of $\phi$ with the projective morphism $ \rho: \Mbl_{\{1\}} \rightarrow \Mpm_{\{1\}}$
\[
\begin{matrix}
\CA_{\bK_1}^{\rm bl} &\longrightarrow & [\Mbl_{\{1\}} / \sG\otimes_{\mathbb{Z}_p}\CO_F] \\
\Big\downarrow && \Big\downarrow \\
\CA_{\bK_1}^\pm &\longrightarrow& [  \RM^\pm_{\Lambda_1}/ \sG\otimes_{\mathbb{Z}_p}\CO_F]
\end{matrix}
\]
where $\Mbl_{\{1\}} $ is the blow-up of $\Mpm_{\{1\}}$ along the unique closed Schubert stratum in the special fiber; see \S \ref{ResolSpin} for the blow-up construction. The scheme $\CA_{\bK_1}^{\rm bl}$ is a linear modification of $\CA_{\bK_1}^\pm$ in the sense of \cite[\S 2]{P}. Thus, from Theorem \ref{ResolutionTheorem} we deduce the following:
\begin{corollary}\label{coro-AKregular}
    The moduli space $\CA_{\bK_1}^{\rm bl}$ of type $D$ has semi-stable reduction and its special fiber is reduced.
\end{corollary}

\subsection{Orthogonal Rapoport-Zink spaces}
Fix an \dfn{integral RZ datum} (cf. \cite[\S 4.1]{RV14})  \begin{flalign*}
	\CD\coloneqq (\BQ_p,B,*,V,\pair{\ ,\ },  \cbra{\mu},[b],\sL)
\end{flalign*} of PEL-type D, where \begin{itemize}
	\item $B=M_2(\BQ_p)$ is a quaternion algebra over $\BQ_p$ with the positive involution $*$ given by $A^*\coloneqq CA^tC$, where $C=\begin{psmallmatrix}
		0 &1\\ -1 &0
	\end{psmallmatrix}$;
	\item $(V,\pair{\ ,\ })$ is a free $B$-module of rank $n+1$, endowed with a perfect alternating $\BQ_p$-bilinear form $\pair{\ ,\ }$ satisfying $\pair{xv,w}=\pair{v,x^*w}$ for $v,w\in V$, $x\in B$;
	\item $\cbra{\mu}$ is a geometric conjugacy class of minuscule cocharacters of $G$, where $G$ is the reductive group over $\BQ_p$ whose $R$-points for any $\BQ_p$-algebra $R$ are given by \begin{flalign*}
    G(R)\coloneqq \cbra{g\in \GL_{B\otimes_{\BQ_p} R}(V\otimes_{\BQ_p} R)\ |\ \pair{gv,gw}=c(g)\pair{v,w} \text{\ for some\ } c(g)\in R\cross }.
        \end{flalign*} 
    \item $[b]\in B(G,\cbra{\mu})$ (see \cite[\S 2.4]{RV14});
    \item $\sL$ is a self-dual $\CO_B$-lattice chain in $V$, where $\CO_B=M_2(\BZ_p)$.
\end{itemize}  

As in \S \ref{ss51}, $G$ is isomorphic to the orthogonal group $\GO(V_1,(\ ,\ ))$. By the classification of PEL-data, there are two possibilities for the (geometric) conjugacy class $\cbra{\mu}$, corresponding to the minuscule coweights $\mu_+$ and $\mu_-$ in \eqref{cochar}.  As in \S \ref{ss51}, one can associate the naive local model $\RM^\naive_\sL$ and spin local model $\RM^\pm_\sL$ to the self-dual lattice chain $\sL$.

Let $\breve \BQ_p$ denote the completion of the maximal unramified extension of $\BQ_p$. Let $\breve\BZ_p$ be the ring of integers of $\breve \BQ_p$. Let $Nilp\coloneqq Nilp_{\breve\BZ_p}$ denote the category of $\breve\BZ_p$-scheme $S$ on which $p$ is locally nilpotent. For $S\in \Nilp$, we denote by $\ol{S}$ the closed subscheme defined by $p\CO_S$. The datum $[b]$ determines a $p$-divisible group $\BX$, which is called a \dfn{framing object}, with $\CO_B$-action $\iota_\BX: \CO_B\ra \End(\BX)$, see \cite[\S 3.20]{RZbook}. 
\begin{defn}
	Let $\CM_\CD^\naive$ denote the functor \begin{flalign*}
		Nilp\ra \Sets
	\end{flalign*}
	sending $S\in Nilp$ to the set of isomorphism classes of the following data. 
	\begin{enumerate}
		\item For each lattice $\Lambda\in \sL$ a $p$-divisible group $X_\Lambda$ over $S$ with an $\CO_B$-action.
		\item For each lattice $\Lambda\in\sL$ a quasi-isogeny \begin{flalign*}
		\rho_\Lambda: \BX\times_{\ol{\BF}_p}\ol{S}\ra X_\Lambda\times_S\ol{S}
		\end{flalign*}
		which commutes with the $\CO_B$-action. 
		\item We require $(X_\Lambda,\rho_\Lambda)_{\Lambda\in\sL}$ satisfies the conditions in \cite[Definition 3.21]{RZbook}.
	\end{enumerate} 
\end{defn}

By \cite[Theorem 3.25]{RZbook}, $\CM_\CD^\naive$ is representable by a formal scheme, which is formally locally of finite type over $\Spf\breve\BZ_p$. This is the naive Rapoport--Zink space. Let $\sG$ denote the group scheme of similitude automorphisms of $\sL$. By results in \cite[\S 3]{RZbook}, there exists a local model diagram of formal schemes over $\Spf\breve\BZ_p$ \begin{flalign}\label{lmd1}
	 \CM^\naive_\CD\xleftarrow{\alpha}\wt{\CM}^\naive_\CD\xrightarrow {\beta}(\RM^\naive_{\sL,\breve\BZ_p})^\wedge,
\end{flalign}
where $\alpha$ is a $\sG_{\breve\BZ_p}$-torsor, $\beta$ is $\sG_{\breve\BZ_p}$-equivariant and formally smooth of relative dimension $\dim G$, and $(\RM^\naive_{\sL,\breve\BZ_p})^\wedge$ denotes the $p$-adic completion of $\RM^\naive_{\sL,\breve\BZ_p}=\RM^\naive_\sL\otimes_{\BZ_p}\breve\BZ_p$.

\begin{defn}
	Let $\CM^\pm_\CD$ denote the functor \begin{flalign*}
		Nilp_{\CO_{\breve F}}\ra \Sets
	\end{flalign*}
	sending $S\in Nilp_{\CO_{\breve F}}$ to the set of isomorphism classes $(X_\Lambda,\rho_\Lambda)_{\Lambda\in \sL}\in \CM^\naive_\CD(S)$ such that the following condition holds. Denote by $M_\Lambda$ the Lie algebra of the universal extension of $X_\Lambda$. Let $\Fil^1_\Lambda\sset M_\Lambda$ denote the Hodge filtration.  Choose an \etale cover $T\ra S$ such that there exists an isomorphism \[\phi: M_\Lambda\otimes_{\CO_S}\CO_T \simeq \Lambda_T\] of $\CO_B\otimes_{\BZ_p}\CO_T$-modules. We require that $\phi(\Fil^1_{\Lambda}\otimes_{\CO_S}\CO_T)$ in $\Lambda_T$ satisfies the spin condition in the sense of condition (v) in Definition \ref{Mpmm}. 
\end{defn}

By the moduli interpretation, the diagram \eqref{lmd1} induces the local model diagram 
\begin{flalign}\label{lmd2}
	 \CM^\pm_\CD\xleftarrow{\hat{\alpha}'}\wt{\CM}^\pm_\CD\xrightarrow {\hat{\beta}'}(\RM^\pm_{\sL,\CO_{\breve F}})^\wedge,
\end{flalign}
where $\hat{\alpha}'$ and $\hat{\beta}'$ satisfy the same properties as in \eqref{lmd1}. In particular, we obtain the following.

\begin{corollary}
	The formal scheme $\CM^\pm_\CD$ is flat over $\CO_{\breve F}$.
\end{corollary}
Therefore, we obtain an explicit moduli interpretation of the orthogonal Rapoport--Zink space in our setting. Lastly, by specializing the lattice chain to $\Lambda_{\{1\}}$ and applying Theorem~\ref{ResolutionTheorem} to the local model diagram (\ref{lmd2}), we obtain a semi-stable formal model of the corresponding Rapoport–Zink space at this level. This is the standard linear modification construction; cf. \cite[\S 7.4]{PaZa}.

\Addresses
\end{document}